\documentclass[10pt]{amsart}

\usepackage{amsfonts,amssymb,bbm, mathrsfs, graphicx}
\usepackage[all]{xy}
\usepackage{enumitem}
\usepackage{hyperref}
\usepackage{mathtools}

\newtheorem{thm}{Theorem}[section]
\newtheorem{prop}[thm]{Proposition}
\newtheorem{cor}[thm]{Corollary}

\theoremstyle{remark}
\newtheorem{remark}[thm]{Remark}

\theoremstyle{definition}
\newtheorem{definition}[thm]{Definition}
\newtheorem{example}[thm]{Example}

\newcommand*\isom{%
  \xrightarrow{\sim}%
}

\def\qq{\mathbb{Q}}
\def\PP{\mathbb{P}}
\def\rr{\mathbb{R}}
\def\zz{\mathbb{Z}}
\def\cc{\mathbb{C}}

\def\mm{\mathcal{M}}

\def\CC{\mathcal{C}}

\def\VV{\mathcal{V}}

\def\a{\mathrm{a}}
\def\vv{\mathbb{V}}
\def\ll{\mathcal{L}}
\def\oo{\mathcal{O}}

\def\ss{\mathcal{S}}
\def\dd{\mathcal{D}}

\def\pp{\mathcal{P}}
\def\ee{\mathcal{E}}
\def\uu{\mathcal{U}}

\def\ra{\rightarrow}

\newcommand{\beq}{\begin{equation}}
\newcommand{\eeq}{\end{equation}}
\newcommand{\pair}[1]{\langle#1\rangle}
\newcommand{\abs}[1]{|#1|}
\newcommand{\divisor}{\operatorname{div}}
\newcommand{\ord}{\operatorname{ord}}

\newcommand{\Hom}{\operatorname{Hom}}

\newcommand{\Spec}{\operatorname{Spec}}

\newcommand{\on}[1]{\operatorname{#1}}
\newcommand{\bb}[1]{{\mathbb{#1}}}
\newcommand{\ca}[1]{{\mathcal{#1}}}

\newcommand{\nm}{\on{N}}
\newcommand{\nmo}{\on{N}^0}

\newcommand\epoints{\partial}

\newcommand\vertices{\mathrm{Vert}}
\newcommand\edges{\mathrm{Ed}}
\newcommand\Graph{\Gamma}
\newcommand\resist{\mu}
\newcommand\current[3]{I(#1:#2\to#3)}
\newcommand\laplace[1][]{L_{#1}}
\newcommand\basis{\mathrm{e}}
\newcommand\transpose[1]{\vphantom{{#1}}^t{#1}}
\newcommand\divD{\mathcal{D}}
\newcommand\divE{\mathcal{E}}
\newcommand\green{g}
\newcommand\tree{T}
\newcommand\trees{\mathrm{Trees}}
\newcommand\forest{F}
\newcommand\forests[1]{#1\text{-}\mathrm{Forests}}

\newcommand\permute[5]{\sigma^{#1#2}_{#3#4}(#5)}
\newcommand\indicator[1]{\chi_{#1}}
\newcommand\maxforests{\mathrm{Max\text{-}Forests}}

\newcounter{nootje}
\numberwithin{equation}{section}

\begin{document}

\title{N\'eron models and the height jump divisor}

\author{Owen Biesel, David Holmes and Robin de Jong}

\subjclass[2010]{Primary 14H10, secondary 11G50, 14G40, 14K15.}

\keywords{Canonical height, Deligne pairing, dual graph, effective resistance, Green's function, height jump divisor, labelled graph, N\'eron model, resistive network.}

\begin{abstract} 
We define an algebraic analogue, in the case of jacobians of curves, of the height jump divisor introduced recently by R. Hain. We give explicit combinatorial formulae for the height jump for families of semistable curves using labelled reduction graphs. With these techniques we prove a conjecture of Hain on the effectivity of the height jump, and also give a new proof of a theorem of Tate, Silverman and Green on the variation of heights in families of abelian varieties. 
\end{abstract}

\maketitle
\thispagestyle{empty}

\tableofcontents

\section{Introduction}

When $S$ is a reduced scheme, and $A$ is an abelian scheme over an open dense subscheme $U$ of $S$, a \emph{N\'eron model} for $A$ over $S$ is a smooth separated group scheme $\nm(A,S) \to S$ together with an isomorphism $j \colon A \isom \nm(A,S)|_U$ satisfying the universal property that for all smooth separated morphisms $T \to S$ and all $U$-morphisms $f \colon T|_U \to A$ there exists a unique $S$-morphism $F \colon T \to \nm(A,S)$ whose restriction to $T|_U$ is equal to $jf$. As was shown by A. N\'eron and M. Raynaud \cite{blr} \cite{ne} \cite{ray} in the 1960s, N\'eron models always exist when the base scheme $S$ is a Dedekind scheme. In contrast, when the base scheme $S$ has dimension at least two, N\'eron models rarely exist, even after allowing alterations of the base. 

If $A$ is a family of jacobian varieties with semistable reduction and $S$ is a quasicompact regular separated base scheme, then by results of the second author \cite{ho1} there exists a largest open subscheme $V \subset S$ containing $U$ such that $A$ does have a N\'eron model $\nm(A,V)$ over $V$. Moreover, the complement of $V$ has codimension at least two in $S$. 

Now given two sections $P, Q \in A(U)$, there exist $m, n \in \zz_{>0}$ such that the multiples $mP, nQ$ extend as sections of the fiberwise connected component $\nmo(A,V)$ of $\nm(A,V)$ over $V$ (perhaps after slightly shrinking $V$). Using a suitable extension of the Poincar\'e bundle on $A \times_U A^\lor$ to the fiberwise connected component of its N\'eron model over $V$ we will construct a canonical $\qq$-line bundle $\pair{P,Q}_\a$ associated to $P, Q$ on $S$, independent of the choice of $m, n$. In the case where $\dim S =1$, the $\qq$-line bundle $\pair{P,Q}_\a$ coincides with the admissible variant, due to S. Zhang \cite{zh}, of the Deligne pairing between degree zero divisors on a semistable curve.

The formation of this $\qq$-line bundle $\pair{P,Q}_\a$ does \emph{not} in general commute with base changes $f \colon T \to S$. In particular, when $T$ is a Dedekind scheme, and $f$ is non-degenerate in the sense that $f^{-1}U$ is dense in $T$, we would like to compare the two $\qq$-line bundles $f^*\pair{P,Q}_\a$ and $\pair{f^*P,f^*Q}_\a$ on $T$. The difference can be viewed as a divisor supported on $f^{-1}(S \setminus V)$, called the \emph{height jump divisor} associated to $P, Q$ and $f$. The terminology is due to R. Hain, who discovered and studied the height jumping phenomenon in a Hodge theoretic context, cf. \cite{hain_normal}. The jump is in fact defined whenever $T$ is a normal scheme, but its computation can always be reduced to the Dedekind case, and so we only consider this (see discussion after Proposition \ref{mostlytrivial}). 

If $A$ has a N\'eron model over $S$, then for every $P, Q$ and $f$ as above the height jump divisor is trivial. In the general case, we apply the theory of labelled graphs (developed by the second author \cite{ho1}) to calculate the height jump divisor in terms of the combinatorics of the dual graphs associated to the fibers of $C \to S$ (cf.\ Theorem \ref{formulahtjump}). Our calculation features the Green's function of the dual graph, viewed as a resistive network. We are then able to give bounds for the coefficients of the height jump divisor by a careful analysis of the Green's function of such networks (cf.\ Theorems \ref{maineffective} and \ref{mainquasisplit}). 

In particular, we show that if $P$ and $Q$ are equal, then the height jump divisor is an effective divisor. The effectivity of the height jump divisor in the diagonal case is related to a conjecture (Conjecture 14.5) in \cite{hain_normal}, which was proven for the tautological sections of the universal jacobian over the moduli stack $\mm_{g,n}$ of $n$-pointed genus-$g$ curves in a previous paper \cite{hdj} by two of the authors. More details will be given in Section \ref{sec:main}.

As a further application of our bounds, we give a new proof of a classical result about the variation of the canonical height $\hat{\on{h}}_\xi$ in a family $(A,\xi)$ of polarized abelian varieties. Assume that $S$ is a smooth projective geometrically connected curve over a number field $K$, and let $U$ be an open dense subscheme of $S$ together with an abelian scheme $A \to U$ over $U$ and a section $P \in A(U)$. Let $\xi$ be a symmetric relatively ample divisor class on $A/U$. Then the aforementioned result states that there exists a $\qq$-line bundle $L$ on $S$ such that $ \deg_S L = \hat{\on{h}}_{\xi_\eta}(P_\eta) $, and such that the function from $U(\bar{K})$ to $\mathbb{R}$ given by $ u \mapsto \hat{\on{h}}_{\xi_u}(P_u)$ extends into a Weil height on $S(\bar{K})$ with respect to $L$ (here $\eta$ denotes the generic point of $S$). 

At the beginning of the 1980s, J. Tate obtained this result for elliptic surfaces over $S$ \cite{tate}, and around the same time J. Silverman \cite{silv} proved an asymptotic version of the general result. Importantly, his result shows that the height of algebraic points $u \in U(\bar{K})$ such that $P_u$ is torsion in $A_u$ is bounded, and hence, by the Northcott property, the set of rational points $u \in U(K)$ such that $P_u$ is torsion in $A_u$, is finite. 

S. Lang \cite[Chapter 12]{lang} and G. Call \cite{call} obtained the result under the assumption of the existence of so-called ``good completions'' of N\'eron models. Finally W. Green proved the general case in \cite{gr}. Green's proof uses the full machinery of toroidal compactifications of the moduli stack of principally polarized abelian varieties. Our alternative proof seems to be somewhat more elementary and reduces the general case of abelian schemes to the special case of jacobians. We do not use any compactifications of semiabelian varieties or of the moduli space of abelian varieties. For the jacobian case, we find ourselves reduced to a study of local height jumps at non-archimedean primes of $K$, where the required bounds are furnished by our general bounds on the coefficients of the height jump divisor.  

\subsection{Outline of the paper}
In Sections \ref{prelims} to \ref{connection} we describe the main results of the paper, and some applications. These results are proven in Sections \ref{proofmain} to \ref{nef}. Sections \ref{resistive} to \ref{basedim1} build up the tools we will need to carry out these proofs. In the appendix we develop some combinatorial tools needed in the proof of the results in Section \ref{basedim1}. 

\section{Preliminaries} \label{prelims}

In this section we collect a number of basic definitions and results. The section can be used as a reference chart for the remainder of the paper, and may be safely skipped at first reading. Our main reference for the material on Poincar\'e biextensions is \cite{mb}, and for the material on Deligne's pairing we refer to \cite{de}.
\begin{definition}[Rigidified line bundles]
Let $U$ be a scheme, let $X \ra U$ be a morphism, and $p \in X(U)$ a section. The category $\on{PicRig}(X,p)$ has as objects pairs $(\ca{L}, \psi)$ where $\ca{L}$ is a line bundle on $X$ and $\psi$ is an isomorphism of line bundles 
\[ \psi\colon p^*\ca{L} \ra O_U \, .  \] 
Morphisms $(\ca{L}_1, \psi_1) \ra (\ca{L}_2, \psi_2)$ are isomorphisms of line bundles $f\colon \ca{L}_1 \isom \ca{L}_2$ such that
$$  (p^*f) \circ \psi_1^{-1}= \psi_2^{-1} \, . $$
\end{definition}
\begin{definition}[The Poincar\'e biextension for abelian schemes]
Let $U$ be a scheme and $A \to U$ an abelian scheme with unit section $e \in A(U)$. The dual abelian scheme $A^\vee$ represents the functor 
\[ \on{PicRig}^0_{A/U, e}\colon  \on{Sch}_U \ra \on{Groups} \]
sending a scheme $T/U$ to the group of isomorphism classes of fibrewise algebraically trivial line bundles on $T \times_U A$ rigidified along the unit section $T \times_U e$. We then define the Poincar\'e bundle $\pp(A)$ on $A \times_U A^\vee$ to be the element corresponding to the identity map in $\on{PicRig}^0_{A/U, e}(A^\vee) = A^\vee(A^\vee)$.
 It is a line bundle on $A \times_U A^\vee$, rigidified along the unit section. It is unique up to unique isomorphism. Moreover, it is an object of $\on{Biext}(A, A^\vee; \bb{G}_m)$, the category of $\bb{G}_m$-biextensions on $A \times_U A^\vee$. 
\end{definition}
\begin{definition}[Admissible metric] \label{def:adm_metric}
Let $U$ be of finite type over $\bb{C}$, let $A \to U$ be an abelian scheme, and let $\ca{L}$ be a line bundle on $A$. An admissible hermitian metric on $\ca{L}$ is a $C^\infty$ hermitian metric on $\ca{L}(\cc)$ whose curvature form is translation invariant in every fiber of $A(\cc) \to U(\cc)$.
\end{definition}  
\begin{prop} \label{prop:adm_metric}
Let $U$ be of finite type over $\bb{C}$, and let $A \to U$ be an abelian scheme. Then the Poincar\'e bundle $\pp(A)$ on $A \times_U A^\lor$ carries a unique admissible metric compatible with its rigidification.
\end{prop}
\begin{proof} See Theorem 3.1 and Definition 3.7 of \cite{mb}. 
\end{proof}
Let $U$ again be any scheme. For a smooth proper curve $C \to U$, and $D$, $E$ two divisors of relative degree zero on $C \to U$, we denote by $\pair{D,E}$ their Deligne pairing \cite{de}. We recall that this is a line bundle on $U$, depending in a bi-additive manner on $D$, $E$. One particularly useful way of thinking about the Deligne pairing is via its connection with the Poincar\'e bundle, as described in the next proposition. For this we need the notation of the jacobian of a smooth proper curve, which we define to be the fiberwise connected component of identity in the relative Picard algebraic space. The jacobian is in fact a scheme, see \cite[Proposition 4.3]{ldg}. 

\begin{prop} \label{delignepoincare}
Let $C \to U$ be a proper smooth curve with jacobian $J \to U$. Let $\mu \colon J \isom J^\lor$ be the canonical principal polarization. Let $D$, $E$ be divisors of relative degree zero on $C \to U$, and write, by a slight abuse of notation, $(D,\mu E)$ for the induced section of $J \times_U J^\vee$.  Then we have a canonical isomorphism
\[ \pair{D,E} \isom (D,\mu E)^* \pp(J)^{\otimes -1} \, , \]
of line bundles on $U$, where $\pp(J)$ is the Poincar\'e bundle on $J \times_U J^\lor$. 
\end{prop}
\begin{proof} Combine Equation 2.9.3 and Corollaire 4.14.1 from \cite{mb}.
\end{proof}
\begin{definition}[Canonical hermitian metric on the Deligne pairing]
\label{metriconNeronpairing} Suppose that $U \times_\bb{Z} \on{Spec}\bb{C}$ is of finite type over $\bb{C}$. Note that by Propositions \ref{prop:adm_metric} and \ref{delignepoincare} we have a canonical metric on $\pair{D,E}$ over $U(\bb{C})$.
\end{definition}
\begin{definition}[Prolonging the Poincar\'e bundle] \label{def:prolong}
Let $S$ be a scheme, and $U \subset S$ a dense open subscheme. Let $G$, $H / S$ be smooth commutative group schemes with connected geometric fibers such that $G_U$ and $H_U$ are dual abelian schemes.
Let $\pp$ be the Poincar\'e biextension on $G_U \times_U H_U$. 
A \emph{Poincar\'e prolongation on $G \times_S H$} is a biextension $\bar{\pp} \in \on{Biext}(G, H; \bb{G}_m)$ extending $\pp$. Such a biextension is unique (up to a unique isomorphism) if it exists.
\end{definition}
\begin{prop}\label{PoincareProlongations}
In the above notation, a Poincar\'e biextension on $G \times_S H$ exists in both of the two following cases:
\begin{enumerate}
\item[(1)] $S$ is a Dedekind scheme;
\item[(2)] $S$ is normal and noetherian, and $G$ and $H$ are semiabelian.
\end{enumerate}
\end{prop}
\begin{proof}
In both situations the restriction functor 
\begin{equation}
\on{res}\colon  \on{Biext}(G,H;\bb{G}_m) \ra \on{Biext}(G_U, H_U;\bb{G}_m)
\end{equation}
is an equivalence of categories. For the first, see Proposition 2.8.2 of \cite{mb}. For the second, see Definition II.1.2.7 and Theorem II.3.6 of \cite{pinc}. 
\end{proof}
\begin{definition}[Semistable curves]\label{def:semistable_curve}
A curve $C$ over a separably
algebraically closed field $k$ will be called \emph{semistable} if $C$ is
connected, reduced, and projective, and each singular point of $C$ is an
ordinary double point. A morphism of schemes $p \colon C \to S$ is a
\emph{semistable curve} if $p$ is proper, flat and finitely presented, and
if all geometric fibers of $p$ are semistable curves. 

For a semistable curve $C \to S$, we write $\mathrm{Sing}(C/S)$ for the
locus where $C \ra S$ is not smooth. More precisely, $\on{Sing}(C/S)$ is cut
out by the first Fitting ideal of the sheaf of relative differentials of $C
\ra S$. We have that $\on{Sing}(C/S)$ is a closed subscheme of $C$ and is
finite unramified over $S$. 

We say that a semistable curve $C \ra S$ is \emph{quasisplit} if
$\mathrm{Sing}(C/S) \to S$ is source-Zariski-locally an immersion and for
every field valued fiber $C_k$ of $C\to S$, every irreducible component of
$C_k$ is geometrically irreducible. Suppose that $S$ is integral, noetherian and regular, and $C \ra S$ is a generically smooth semistable curve. Then there exists a surjective \'etale morphism $S' \to S$ such that the semistable curve $C \times_S S' \to S'$ is quasisplit. This is not hard to prove; it is enough to show that every geometric point $\bar{s} \ra S$ factors via an \'etale map $S' \ra S$ with $C \times_S S' \to S'$ quasisplit. It is clear that this latter property is satisfied after base change to the spectrum of the \'etale local ring of $S$ at $\bar{s}$. By writing the spectrum of the \'etale local ring as a filtered limit of \'etale covers one obtains the result by a finite presentation argument - details can be found in \cite[Lemma 4.3]{ho2}. Later (in Proposition \ref{lem:quasi_split_after_alteration}) we will see that the
same holds for some alteration $S' \ra S$ (by `alteration', we mean a proper surjective generically finite morphism). This will be important in the reduction steps in Section \ref{variation}. 

\end{definition}
\begin{definition}[$\qq$-line bundles] When $S$ is a scheme, the category of $\qq$-line bundles on $S$ has as
objects pairs $(m,\ll)$ with $m$ a positive integer and $\ll$ a line bundle
on $S$. The hom-set $\Hom\left((m_1,\ll_1),(m_2,\ll_2)\right)$ is to be the set of equivalence classes of pairs $(a,f)$
with $a$ a positive integer, and $f \colon \ll_1^{\otimes
am_2} \to \ll_2^{\otimes am_1}$ a homomorphism of line bundles on $S$. The equivalence relation is generated by setting $(a,f) \sim (an,f^{\otimes n})$ for each $n\in\zz_{>0}$.
For a class $[(a,f)]$ in
$\Hom((m_1,\ll_1),(m_2,\ll_2))$ and $[(b,g)]$ in
$\Hom((m_2,\ll_2),(m_3,\ll_3))$ the composition $[(b,g)] \circ [(a,f)]$ is
defined as the class of the pair $(abm_2,h)$ where $h \colon \ll_1^{\otimes
abm_2m_3} \to \ll_3^{\otimes abm_1m_2}$ is the composition $g^{\otimes am_1}
\circ f^{\otimes bm_3}$. One verifies that this composition law is associative and independent of the choice of representatives, and that $[(1,\mathrm{id}_{\ll})]$ acts as the identity morphism for $(m, \ll)$. 

A global section of a $\qq$-line bundle $(m,\ll)$ is to be a morphism from $(1,O_S)$ to $(m,\ll)$, represented by a pair $(a,s)$ with $s$ a global section of $\ll^{\otimes a}$. A rational section of $(m,\ll)$ is a global section of $(m,\ll_U)$ for some open dense subscheme $U \subset S$. Isomorphisms are morphisms with two-sided inverses; for example, for non-zero integers $m, n$ the $\qq$-line bundles $(m,\ll)$ and $(mn,\ll^{\otimes n})$ are canonically isomorphic. We will often use the notation $\ll^{\otimes 1/m}$ to denote the $\qq$-line bundle $(m,\ll)$.

We leave it to the reader to verify the following facts. There is an (obvious) notion of tensor product of $\qq$-line bundles; in fact, the set of isomorphism classes of $\qq$-line bundes on $S$ is naturally a $\qq$-vector space, and is canonically isomorphic to $\qq$ tensored with the abelian group of isomorphism classes of ordinary line bundles on $S$.
There is an (obvious) notion of pullback of $\qq$-line bundles. 
A $\qq$-Cartier divisor $D$ on $S$ naturally gives rise to a $\qq$-line bundle $O_S(D)$ on $S$: if $m$ is a non-zero integer such that $mD$ is a Cartier divisor, then one sets $O_S(D)=(m,O_S(mD))$. Vice versa, to a rational section of a $\qq$-line bundle one can naturally associate its divisor; this is a $\qq$-Cartier divisor on $S$.
\end{definition}

\section{Admissible pairing and the height jump divisor}

The purpose of this section is to introduce the main objects of our work: the admissible pairing, and the height jump divisor.

Let $S$ be an integral, noetherian, regular and separated scheme, and let $p \colon C \to S$ be a semistable curve. We assume that $p$ is generically smooth; let $U \subset S$ be the largest open subscheme such that the restriction $C_U \to U$ of $C$ to $U$ is smooth. Let $J_U \to U$ be the jacobian scheme associated to the smooth curve $C_U \to U$. Denote by $\pp=\pp(J_U)$ the Poincar\'e bundle on $J_U \times_U J_U^\lor$ and denote by $\mu \colon J_U \isom J^\lor_U$ the canonical principal polarization.

A first important ingredient in our construction is the following result.
\begin{thm} \label{existenceV} 
\leavevmode
\begin{enumerate}[label=(\alph*)]
\item
There exists a maximal open subscheme $V \subset S$ with the following properties: (a) $U \subset V$; (b) the jacobian scheme $J_U \to U$ extends into a N\'eron model $
\nm(J_U) \to V$ over $V$. Moreover, the codimension of the complement of $V$ in $S$ is at least two.
\item Let $\sigma\colon V \ra \nm(J_U)$ be a section. Then there exists an integer $n > 0$ and an open subset $U \subset V_\sigma \subset V$ such that over $V_\sigma$ the section $n\sigma$ is contained in the identity component of the N\'eron model $\nm(J_U)$, and the codimension of the complement of $V_\sigma$ in $S$ is at least two. 
\end{enumerate}
\end{thm}
In fact the open $V_\sigma$ may be taken to equal $V$ since the N\'eron model can be shown to be of finite type, but we will not need this here. 
\begin{proof}
Part (a) is \cite[Corollary 1.3]{ho1}. For part (2), by a limiting argument we find that formation of the N\'eron model commutes with pullback to the spectrum of the local ring at the generic point of a boundary divisor of $U$ in $S$. Such a local ring is Dedekind, so the N\'eron model over it is of finite type (\cite[1.2]{blr}). Since $U$ has finitely many boundary components we can find a positive integer $n$ such that $n\sigma$ lies in the fiberwise connected component of the identity $\nm^0$ of $\nm\coloneqq \nm(J_U)$ over the generic point of every boundary divisor of $U$ in $S$. Now let $V_\sigma = (n\sigma)^{-1}\nm^0$, then $V_\sigma$ is open in $S$ and contains the generic point of every boundary divisor of $U$ in $S$, and hence the codimension of $V_\sigma$ in $S$ is at least two. 
\end{proof}

In view of the isomorphism $\mu \colon J_U \isom J^\lor_U$ we also have a N\'eron model $\nm(J_U^\lor) \to V$ of $J_U^\lor$ over $V$, and by the N\'eron mapping property the isomorphism $\mu$ extends into an isomorphism $\bar{\mu} \colon \nm(J_U) \to \nm(J^\lor_U)$ over $V$.
Let $\nm^0(J_U)$ be the fiberwise connected component of the N\'eron model $
\nm(J_U) \to V$ furnished by Theorem \ref{existenceV}, and define $\nm^0(J^\lor_U)$ in a similar way. As $p \colon C \to S$ is semistable we have that $\nm^0(J_U)$ and $\nm^0(J_U^\lor)$ are semiabelian schemes over $V$. Let $\bar{\pp}$ be the unique biextension line bundle on the product $
\nm^0(J_U) \times_V \nm^0(J_U^\lor)$ extending the Poincar\'e bundle $\pp$ on $J_U \times_U J_U^\lor$ (cf.\ Proposition \ref{PoincareProlongations}). 

We fix two relative Cartier divisors $D, E$ on $C$ which are of relative degree zero over $S$. The divisors $D, E$ give rise to two sections $D_U$ resp. $\mu E_U$ of $J_U/U$ resp $J^\lor_U/U$. By the N\'eron mapping property, both $D_U$ and $ \mu E_U$ extend as sections - which we shall denote by $D, \bar{\mu} E$ - of $\nm(J_U)$ resp. $\nm(J_U^\lor)$ over $V$. By part (b) of Theorem \ref{existenceV}, after shrinking $V$ there exist $m, n \in \zz_{>0}$ such that $mD_U, n\mu E_U$ extend as sections $mD$ resp. $n\bar{\mu} E$ of $
\nm^0(J_U)$ resp $\nm^0(J^\lor_U)$ over $V$. Take such $m, n \in \zz_{>0}$. Then we define the line bundle 
\[ \ll(m,n;D,E)=(mD,n\bar{\mu} E)^*\bar{\pp}^{\otimes -1} \] 
on $V$. As the complement of $V$ has codimension at least two in $S$, and as $S$ is regular, the line bundle $\ll(m,n;D,E)$ extends uniquely as a line bundle over $S$. We will denote this line bundle by $\bar{\ll}(m,n;D,E)$. 
\begin{definition}[Admissible pairing] We define $\pair{D,E}_\a$ to be the $\qq$-line bundle $\bar{\ll}(m,n;D,E)^{\otimes 1/mn}$ on $S$. Note that as $\bar{\pp}$ is a biextension, different choices of $m$ and $n$ in
the definition of $\pair{D,E}_\a$ lead to canonically isomorphic $\qq$-line bundles. Furthermore, the formation of $\pair{D,E}_\a$ is biadditive on divisors $D, E$ of relative degree zero. We call $\pair{D,E}_\a$ the \emph{admissible pairing} associated to $D, E$ on $S$.
\end{definition}
The terminology `admissible pairing'  will be justified later (Section \ref{basedim1}) when we show that in the case when $S$ has dimension one, it coincides with S. Zhang's admissible pairing \cite{zh}. 

As explained in the introduction, we are interested in the extent to which the formation of $\pair{D,E}_\a$ is compatible with base change. Let $f \colon T \to S$ be a morphism of schemes with $T$ also integral, noetherian, regular and separated, and such that $f^{-1}U$ is dense in $T$. We call such morphisms \emph{non-degenerate}. Let $C_T = C \times_S T \to T$ be the pullback of $p \colon C \to S$ along $f$; note that this is a semistable curve over $T$. Applying the construction above to the pulled back divisors $f^*D, f^*E$ on $C_T$, we obtain a natural $\qq$-line bundle $\pair{f^*D,f^*E}_\a$ associated to $D, E$ and $f$ on $T$. When restricted to $f^{-1}U$, the $\qq$-line bundles $f^*\pair{D,E}_\a$ and $\pair{f^*D,f^*E}_\a$ are canonically isomorphic. Hence we obtain a canonical non-zero rational section $\sigma(f;D,E)$ of the $\qq$-line bundle
\[ f^*\pair{D,E}_\a^{-1} \otimes \pair{f^*D,f^*E}_\a \]
on $T$ supported on $T \setminus f^{-1}U$.
\begin{definition}[Height jump divisor] We define $J=J(f;D,E)$ to be the divisor of $\sigma(f;D,E)$ on $T$. It is a $\qq$-Cartier divisor on $T$. We call $J(f;D,E)$ the \emph{height jump divisor} associated to $D, E$ and $f$.
\end{definition}
In many cases, the height jump divisor is trivial.
\begin{prop} \label{mostlytrivial}
Assume $f \colon T \to S$ is a flat non-degenerate morphism. Then the formation of $\pair{D,E}_\a$ is compatible with base change along $f$.
\end{prop}
\begin{proof} Recall that $f$ being non-degenerate implies that $T$ is integral, noetherian, regular and separated. The formation of  $\pair{D,E}_\a$ is compatible with base change along $f$ if and only if the jump divisor $J(f;D,E)$ is trivial. The triviality of a divisor can be checked on the complement of a codimension 2 subscheme, so we may assume that the image of $f$ is contained in $V$. 
Now the sections $mD$ and $nE$ of the N\'eron model $
\nm(J_U)$ are contained in the fibrewise connected component of identity $\nm^0(J_U)$, by assumption. The pullback of $\nm^0(J_U)$ along $f$ is again semiabelian, and prolongs the pullback of the jacobian. Moreover, the sections $mD$ and $nE$ pull back to sections of $f^*\nm^0(J_U)$, and the rigidified extension of the Poincar\'e bundle pulls back to a rigidified extension of the Poincar\'e bundle. The result then follows from the uniqueness of semiabelian prolongations, cf.\ Theorem 1.2 in \cite{ldg}.
\end{proof}
More generally, one can compute the multiplicity of the height jump divisor $J$ along a prime divisor $Z$ in $T$ using the `test curve' which is the canonical map from the spectrum of the local ring of $T$ at the generic point of $Z$ to $T$ itself. It follows that in order to study height jumping we can mostly restrict ourselves to the case where $T$ is the spectrum of a discrete valuation ring (i.e., a trait). In this case, the height jump divisor is a rational multiple $J = j \cdot [t]$ of the closed point $t$ of $T$. 

\section{Statement of the main results} \label{sec:main}

As above, let $S$ be an integral, noetherian, regular, separated scheme, and let $p \colon C \to S$ be a semistable curve. Let $U \subset S$ be the largest open subscheme over which $p$ is smooth, and assume $U$ is dense in $S$. Denote by $Z$ the reduced closed subscheme of $S$ determined by the complement of $U$ in $S$. As $S$ is noetherian $Z$ has only finitely many irreducible components; we write $Z = \cup_{i=1}^r Z_i$ for the decomposition of $Z$ into its irreducible components. Note that each $Z_i$ is a (prime) divisor, as one can see from the local structure of semistable curves. 

\subsection{Effectivity of the height jump divisor}

Our first result states that in the case where $D$ and $E$ are chosen to be equal, the height jump divisor $J(f;D,E)$ is effective.
\begin{thm} \label{maineffective}
Let $D$ be a divisor of relative degree zero on $C \to S$, with support contained in the smooth locus $\mathrm{Sm}(C/S)$ of $C \to S$. Let $s \in S$ be a point. Then for all non-degenerate morphisms of pointed schemes $ f \colon (T,t) \to (S,s)$ with $(T,t)$ a trait, the height jump divisor $J(f;D,D)$ is an effective divisor on $T$.
\end{thm}
We give a proof of Theorem \ref{maineffective} in Section \ref{proofmain}.

\subsection{Bounds for the height jump divisor}

Assuming that the semistable curve $p \colon C \to S$ is quasisplit, the next result describes the behavior of the height jump divisor as the test morphism $f \colon T \to S$ varies. There turns out to be a uniform description of the height jumps in terms of a certain rational function $\Phi$, homogeneous of weight one, associated to the base point $s \in S$ and the divisors $D, E$. Moreover, we can control the ``growth behavior'' of this function $\Phi$ in terms of a more manageable homogeneous weight one function. 
\begin{thm} \label{mainquasisplit}
Assume that $p \colon C \to S$ is quasisplit. Let $D, E$ be two divisors of relative degree zero on $C \to S$ with support contained in $\mathrm{Sm}(C/S)$. Let $s \in S$ be a point. Whenever $f \colon (T,t) \to (S,s)$ is a non-degenerate morphism of schemes with $(T,t)$ a trait, we write $m_i$ for $\ord_{t} f^*Z_i$ for $i=1,\ldots,r$.
\begin{enumerate}[label=$(\alph*)$]
\item There exists a unique rational function $\Phi \in \qq(x_1,\ldots,x_r)$ such that for all non-degenerate morphisms of pointed schemes $ f \colon (T,t) \to (S,s)$ with $(T,t)$ a trait, the equality
\[ \ord_t J(f;D,E) = \Phi(m_1,\ldots,m_r) \]
holds. The function $\Phi$ is homogeneous of weight one and has no linear part in the sense that for each $i=1,\ldots,r$ we have $\Phi(0,\ldots,0,1,0,\ldots,0)=0$ where the $1$ is placed at the $i$-th spot. 
\item There exists a constant $c$ such that for all non-degenerate morphisms of pointed schemes $ f \colon (T,t) \to (S,s)$ with $(T,t)$ a trait, the bound\[ | \ord_{t} J(f;D,E) | \leq c \cdot \min_{i=1,\ldots,r} (\sum_{j \neq i} m_j) \]
holds.
\end{enumerate}
\end{thm}
The homogeneous weight one function $\Phi$ from part (a) of the theorem will be made explicit in terms of the combinatorics of the dual graph and the singularities of $C$ lying above $s$; see Theorem \ref{formulahtjump} below. The constant $c$ from part (b) will also be made effective. Part (b) shows in particular that if $s$ lies in at most one of the $Z_i$, all height jumps are trivial. In other words, under the conditions of the theorem, the height jump divisor $J(f;D,E)$ is supported on $f^{-1}Z'$, where $Z'=\cup_{i\neq j} Z_i \cap Z_j$ is the union of the mutual intersections of the $Z_i$. Our proof of Theorem \ref{mainquasisplit} will also be given in Section \ref{proofmain}.

\subsection{Variation of the canonical height}

We would like to discuss two applications of Theorems \ref{maineffective} and \ref{mainquasisplit} above. Our first one, as already amply discussed in the introduction, concerns the variation of the canonical height of a section of a family of polarized abelian varieties. 

Let $K$ be a number field or the function field of a curve over a field, and fix an algebraic closure $K \subset \bar{K}$ of $K$. W. Green proved the following theorem in \cite{gr} in the number field case.
\begin{thm} \label{greenlocal} Let $S$ be a smooth projective geometrically connected curve over $K$, with generic point $\eta$. Let $U$ be an open dense subscheme of $S$ together with an abelian scheme $A \to U$ over $U$ and a section $P \in A(U)$. Let $\xi$ be a symmetric relatively ample divisor class on $A \to U$. Then there exists a $\qq$-line bundle $L$ on $S$ such that $ \deg_S L = \hat{\on{h}}_{\xi_\eta}(P_\eta) $ and such that the function $U(\bar{K}) \to \rr$ given by $ u \mapsto \hat{\on{h}}_{\xi_u}(P_u)$ extends into a Weil height on $S(\bar{K})$ with respect to $L$. 
\end{thm}
Using our bound in Theorem \ref{mainquasisplit}(b), we are able to give an alternative proof of Green's theorem. Our proof of Theorem \ref{greenlocal} will be given in Section \ref{variation}.

\subsection{A nefness result on the moduli of pointed stable curves} \label{sub:nef}In Section \ref{nef} we will prove the following theorem. 
\begin{thm} \label{nefnessproperty}
Assume that $S$ is of finite type over a field $k$. Let $D$ be a divisor of relative degree zero on $C \to S$ with support contained in $\mathrm{Sm}(C/S)$. Let $T$ be a smooth projective geometrically connected curve over $k$, and let $f \colon T \to S$ be a non-degenerate $k$-morphism. Then the $\qq$-line bundle $f^*\pair{D,D}_\a^{\otimes -1}$ has non-negative degree on $T$.
\end{thm}

As a special case we find that a certain line bundle on the moduli stack $\overline{\mm}_{g,n}$ of $n$-pointed stable curves of genus $g$ has non-negative degree on all complete curves that do not lie in the boundary divisor (i.e. a weak form of nefness). This issue is also discussed in \cite{hain_normal}. Our result is related with the discussion following Conjecture 14.5 in \cite[Section 14]{hain_normal}, as will follow from our next Section \ref{connection} on the connection of our work with Hain's.

\begin{cor}  Let $k$ be a field, and let $\overline{\mm}_{g,n}$ be the moduli stack of $n$-pointed stable curves of genus $g$ over $k$. Let $(p \colon \overline{C}_{g,n} \to \overline{\mm}_{g,n},(x_1,\ldots,x_n))$ be the universal pointed stable curve, and let $D = \sum m_i x_i$ be a relative degree zero divisor supported on the $x_i$. Let $f \colon T \to \overline{\mm}_{g,n}$ be a non-degenerate morphism (with respect to the universal curve) with $T$ a smooth projective curve over $k$. Then the $\qq$-line bundle $f^*\pair{D,D}_\a^{\otimes -1}$ has non-negative degree on $T$. 
\end{cor}

With the same methods it is possible to prove an analogue of Theorem \ref{nefnessproperty} in Arakelov geometry:
\begin{thm}\label{nefnessproperty_Arakelov}
Assume that $S$ is proper and flat over $\on{Spec} \bb{Z}$. Let $D$ be a divisor of relative degree zero on $C/S$ with support contained in $\on{Sm}(C/S)$. Let $T=\on{Spec} \ca{O}$ with $\ca{O}$ the ring of integers in a number field, and let $f \colon T \ra S$ be morphism, non-degenerate with respect to $C \ra S$. Then the hermitian $\bb{Q}$-line bundle $f^*\pair{D,D}_a^{-1}$ on $T$ has non-negative Arakelov degree.
\end{thm}

\section{Connection with Hain's work}  \label{connection}

Before continuing, we would like to point out the relation with Hain's work \cite{hain_normal}, which takes place in a Hodge theoretic context. The material from this section will not be used in what follows. Our main references for this section are \cite{hainbiext} and \cite{hain_normal}.

Let $U$ be an integral, regular and separated scheme of finite type over $\cc$. Let $(\vv,\mu)$ be an admissible variation of polarized Hodge structures of weight $-1$ over  $U(\cc)$, and $(\vv^\lor,\mu^\lor)$ its dual. Let $J(\vv) \to U(\cc)$ be the intermediate jacobian fibration associated to $\vv$, with dual $J(\vv^\lor)$. Then the torus fibration $J(\vv) \times_{U(\cc)} J(\vv^\lor)$ carries a canonical biextension line bundle $\pp=\pp(J(\vv))$, equipped with a canonical (admissible) $C^\infty$ hermitian metric.

Now suppose we have a section $\nu \colon U(\cc) \to J(\vv)\times_{U(\cc)} J(\vv^\lor)$. We then obtain a $C^\infty$ hermitian line bundle $\ll=\nu^* \pp(J(\vv))$ on $U(\cc)$. Now assume $S$ is a partial compactification of $U$ such that $S \setminus U$ is a normal crossings divisor $Z = \sum_{i=1}^r Z_i$ and $\vv$ has unipotent monodromy around each of the $Z_i$. In his PhD thesis \cite{lear}, D. Lear shows that there exists a unique $\qq$-line bundle $\bar{\ll}$ on $S$ extending the line bundle $\ll$ on $U$ in such a way that the canonical metric on $\ll$ extends into a continuous metric on the restriction of $\bar{\ll}$ to $S \setminus Z^{\mathrm{sing}}$. We call $\bar{\ll}$ the \emph{Lear extension} of $\ll$ over $S$. 

In \cite{hain_normal} Hain studies and computes the Lear extension in a number of examples related to moduli spaces of pointed curves. For $f \colon T \to S$ a non-degenerate morphism, with $T$ a smooth projective curve over $\cc$, one can compare the pullback $f^*\bar{\ll}$ of the Lear extension with the Lear extension on $T$ obtained from the pullback section $f^*\nu$ and the pullback variation of Hodge structures $f^*\vv$, leading to a height jump divisor $J=J(f;\nu)$ supported on $T \setminus f^{-1}U$. Hain conjectures in \cite[Section 14]{hain_normal} that the height jump divisor should be effective in the ``diagonal'' case where $\nu$ maps into the graph of the given polarization $\mu \colon J(\vv) \to  J(\vv^\lor)$.

The special case connected to the theme of the present paper is the case where $(\vv,\mu)$ is the polarized variation of Hodge structures on $U(\cc)$ associated to a semistable curve $p \colon C \to S$, where $U$ is the locus where $p$ is smooth. That is, the fibre of $\bb{V}$ at a point $u \in U$ is $\on{H}_1(C_u)$ endowed with its canonical principal polarisation. The intermediate jacobian fibration associated to $\vv$ is then the analytification of the jacobian $J_U \to U$ of $C_U \to U$. The section $\nu \colon U(\cc) \to J(\vv)\times_{U(\cc)} J(\vv^\lor)$ is the section $(D,\mu E)$ determined by a pair of relative degree zero divisors $D, E$ on $C \to S$ with support contained in $\mathrm{Sm}(C/S)$. 

Let $\nm(J_U) \to V$ be the N\'eron model of $J_U \to U$ furnished by Theorem \ref{existenceV}. It follows from \cite[Section 4]{gr} that the Poincar\'e prolongation $\bar{\pp}=\bar{\pp}(J_U)$ on $
\nm^0(J_U) \times_U \nm^0(J_U^\lor)$ can be endowed with a continuous hermitian metric extending the canonical (admissible) $C^\infty$ hermitian metric on $\pp(J_U)$ (cf.\ Definition \ref{def:adm_metric} and Proposition \ref{prop:adm_metric}). 

We deduce from this that our admissible pairing and Lear's extension coincide.
\begin{thm} Let $p \colon C \to S$ be a generically smooth semistable curve, and let $U \subset S$ be the locus where $p$ is smooth. Let $D, E$ be two divisors on $C \to S$ of relative degree zero, and $\pair{D,E}_\a$ be their admissible pairing on $S$. Then $\pair{D,E}_\a^{\otimes -1}$ coincides with the Lear extension of the $C^\infty$-hermitian line bundle $\nu^*\pp(J_U)$ on $U$, where $\nu \colon U \to J_U \times_U J_U^\lor$ is the section determined by the restriction of the pair $(D,\mu E)$ to $U$.
\end{thm}
Using this equivalence, one sees that Theorem \ref{maineffective} proves and generalizes a conjecture of Hain about the effectivity of the height jump divisor for Lear extensions in \cite[Section 14]{hain_normal}. Furthermore,  one can now also see that the result in Theorem \ref{mainquasisplit}(a) is an algebraic version of an analytic result due to G. Pearlstein \cite[Theorems 5.19 and 5.37]{pearl}, if the latter is specialised to variations of Hodge structures of type $(-1,0), (0,-1)$. Finally, referring back to our results in subsection \ref{sub:nef}, Theorem 11.5 of \cite{hain_normal} calculates, for a given tuple $(m_1,\ldots,m_n)$ of integers such that $\sum_i m_i =0$, the admissible pairing $\pair{D,D}_\a^{\otimes -1}$ on $\overline{\mm}_{g,n}$ with $D = \sum_i m_i x_i$. In Section 9 of \cite{hdj} one finds an alternative calculation, more in the spirit of our ``algebraic'' approach.

\section{Resistive networks and Green's functions} \label{resistive}

Our proofs of Theorems \ref{maineffective} and \ref{mainquasisplit} rely on an explicit formula for the height jump divisor in the quasisplit case. The objective of the next sections will be to develop the necessary preliminary results in order to state this formula (Theorem \ref{formulahtjump}) and to prove this formula (Section \ref{explicit}).

An important tool is the Green's function on a resistive network. 
This will be the subject of the present section. 
In the next section we will recall from \cite{ho1} and \cite{ho2} the notion of a labelled dual graph for a point in the base $S$ of a semistable curve $C \to S$, and state Theorem \ref{formulahtjump} in terms of these labelled dual graphs. 

\begin{definition}[Graphs]
A \emph{graph} is a triple $(V, E, \epoints)$, where $V$ and $E$ are sets (the set of \emph{vertices} an the set of \emph{edges}, respectively), and $\epoints: E\to (V\times V)/\mathrm{S}_2$ is a function sending each edge to its unordered pair of endpoints.  
Thus we allow parallel edges (multiple edges sharing the same set of endpoints) and loops (edges whose ``two'' endpoints are equal). An \emph{orientation} of an edge is an ordering of its endpoints. An \emph{oriented edge} is an edge equipped with an orientation. If $e$ is an edge with $\epoints(e)=[(i,j)]$, we refer to its orientations as $e:i\to j$ and $e:j\to i$. Given a graph $\Graph$, we refer to its set of vertices by $\vertices(\Graph)$, and its set of edges by $\edges(\Graph)$. 
\end{definition}

\begin{definition}
A \emph{resistive network} is to be a pair $(\Graph,\resist)$ where $\Graph$ is a graph with finite sets of vertices and edges, and $\resist \in \rr_{\geq 0}^{\edges(\Graph)}$ is a function assigning a nonnegative real number to each edge of $\Graph$.
We say that an edge $e$ of a resistive network $(\Graph,\resist)$ has a \emph{resistance} of $\resist(e)$.  
In case the resistance of each edge is strictly positive, we say that the resistive network is \emph{proper}; a resistive network where some of the edges have zero resistance is called \emph{improper}.
\end{definition}
\begin{definition}Each proper resistive network $(\Graph,\resist)$ has an associated \emph{Laplacian} $\laplace=\laplace[(\Graph, \resist)]$, a linear map $\rr^{\vertices(\Graph)}\to\rr^{\vertices(\Graph)}$.  Given a vector $v=(v_i)_{i\in \vertices(\Graph)}\in \rr^{\vertices(\Graph)}$,
the $i$th component of $\laplace v$ is given by 
\[(\laplace v)_i = \sum_{j\in\vertices(\Graph)}\sum_{\substack{\text{edges}\\e:i\to j}}\frac{v_i - v_j}{\resist(e)} \, .\]
It is straightforward to check that $\laplace$ is self-adjoint (i.e.\ $\transpose{\laplace}=\laplace$), that the kernel of $L$ consists of vectors which are constant on each connected component of $\Graph$, and that the image of $\laplace$ consists of vectors that sum to zero on each connected component.
\end{definition}
%

\begin{remark}
 The connection with electrical resistances is to interpret a vector $v\in\rr^{\vertices(\Graph)}$ as an assignment of a real-valued \emph{voltage} to each vertex in $\Graph$.
 Then for each edge $e:i\to j$, the quantity $(v_i-v_j)/\resist(e)$ is interpreted by Ohm's law as the \emph{current} flowing along edge $e$.
 The above formula for $(\laplace v)_i$, then, calculates the total current flowing out of vertex $i$ into the rest of the network.
 We say that $\laplace v$ is the \emph{(vertex) current assignment} induced by $v$.
 We will typically denote vectors in $\rr^{\vertices(\Graph)}$ by small italic letters $v$, $w$, etc.\ if they are to be interpreted as voltage assignments, or by large calligraphic letters $\divD$, $\divE$, etc.\ if they are to be interpreted as current assignments.
\end{remark}

For the remainder of this section, we will consider only those resistive networks which are connected, that is, which have exactly one connected component.
Then the kernel of $\laplace$ consists of the constant vectors, and the image of $\laplace$ consists of the vectors $\divD$ whose sum is zero.
Thus by dimensional considerations, $\laplace$ restricts to a linear automorphism of the vector space of zero-sum vectors in $\rr^{\vertices(\Graph)}$.
Its inverse extends uniquely to a linear endomorphism of $\rr^{\vertices(\Graph)}$ whose kernel also consists of the constant vectors; this endomorphism $\laplace^+$ is called the \emph{Moore-Penrose pseudoinverse} of $\laplace$, and it is also self-adjoint.
Given a vector $\divD$ in $\rr^{\vertices(\Graph)}$, we may compute  $\laplace^+\divD$ by first adding a constant vector to $\divD$ to make the sum of its entries vanish, then finding a voltage assignment $v$ inducing that zero-sum current assignment, and finally adding a constant vector to $v$ to make the sum of \emph{its} entries vanish.
If the sum of the entries of $\divD$ already vanishes, we may omit the first step, and if we are only interested in the differences between entries of $\laplace^+\divD$, then the last step may be omitted as well.
In this way, we can speak of voltage differences induced by a current assignment $\divD$: if $v$ is any vector with $Lv=\divD$, then $v_i-v_j = (\laplace^+\divD)_i - (\laplace^+\divD)_j$.

Given two vertices $i$ and $j$ in a resistive network $(\Graph, \resist)$, the \emph{effective resistance} $r_\mathrm{eff}(i,j)$ from $i$ to $j$ is the voltage difference between vertices $i$ and $j$ when a current of $+1$ is imposed at vertex $i$ and $-1$ is imposed at vertex $j$ (and $0$ everywhere else).  
Denoting by $\basis_k$ the vector with $1$ in the $k$th place and $0$ everywhere else, we can write this current assignment as $\basis_i-\basis_j$.
Then $\laplace^+(\basis_i-\basis_j)$ is a voltage assignment inducing such a current, and $r_\mathrm{eff}(i,j)=\transpose{(\basis_i-\basis_j)}\laplace^+(\basis_i-\basis_j)$ is the resulting voltage difference from vertex $i$ to vertex $j$.
More generally, given two zero-sum vectors $\divD,\divE\in\rr^{\vertices(\Graph)}$, we define the Green's function for $(\Graph,\resist)$ as follows:
\begin{definition}
 Let $(\Graph, \resist)$ be a proper resistive network whose underlying graph $\Graph$ has exactly one connected component, and let $\divD$ and $\divE$ be two zero-sum vectors in $\rr^{\vertices(\Graph)}$.
 Then the \emph{Green's function} of $(\Graph, \resist)$ at $\divD$ and $\divE$ is defined as
\[\green(\Graph, \resist; \divD, \divE):=\transpose{\divD}\laplace^+\divE\]
where $\laplace^+$ is the
Moore-Penrose pseudoinverse to the Laplacian $\laplace=\laplace[(\Graph,\resist)]$.
Then for fixed $\divD$ and $\divE$, we may consider $\green(\Graph,\cdot\,; \divD, \divE)$ to be a function $\rr_{>0}^{\edges(\Graph)}\to \rr$.  
Alternatively, we may fix $\resist$ and consider $\green(\Graph,\resist; \cdot\,, \cdot\,)$ as a symmetric bilinear form.
\end{definition}

In the appendix we will use the techniques of resistor networks to prove that the Green's function extends continuously to improper networks:

\begin{prop} \label{continuous}
Let $\Graph$ be a connected graph, and $\divD$ and $\divE$ two zero-sum vectors in $\rr^{\vertices(\Graph)}$. The Green's function $g(\Graph,\cdot\,;\divD,\divE)$ extends continuously to a function $\rr_{\geq 0}^{\edges(\Graph)}\to\rr$.
\end{prop}

We can even write down what the Green's function is for an improper network $(\Graph, \resist_0)$.
Let $S\subset\edges(\Graph)$ be the set of edges $e$ whose resistances $\resist_0(e)$ vanish, and let $\Graph/S$ be the graph obtained from contracting the edges in $S$ (i.e.\ identifying the two endpoints of each edge in $S$ and then removing those edges).
Thus the edges of $\Graph/S$ are naturally identified with the edges of $\Graph$ not in $S$, so by restricting $\resist_0$ we obtain a proper resistance network structure on $\Graph/S$.
Each vertex of $\Graph/S$ corresponds to an equivalence class of vertices of $\Graph$, so we have a surjection $[~\cdot~]:\vertices(\Graph)\to\vertices(\Graph/S)$ sending each vertex $i$ to its equivalence class. 
This surjection extends to an $\rr$-linear map $[~\cdot~]:\rr^{\vertices(\Graph)}\to\rr^{\vertices(\Graph/S)}$ via $[\basis_i]=\basis_{[i]}$.
We prove Proposition~\ref{continuous} by showing that the Green's function on this new graph is precisely the limit of the Green's function on the original:
\begin{equation} \label{limiting-value}
 \lim_{\substack{\resist\to\resist_0\\\resist\text{ proper}}}\green(\Graph,\resist; \divD, \divE) = \green\bigl(\Graph/S,\resist_0|_{\edges(\Graph/S)};[\divD],[\divE]\bigr) 
\end{equation}
In particular, the limit on the left-hand side exists, so $\green(\Graph, \cdot \,; \divD,\divE)$ extends continuously to all of $\rr_{\geq 0}^{\edges(\Graph)}$.

We also prove the following facts about Green's functions:

\begin{prop}\label{basicgreenfacts}
Let $\Graph$ be a connected graph, and let $\divD$ and $\divE$ be zero-sum elements of $\rr^{\vertices(\Graph)}$. 
\begin{enumerate}[label=(\alph*)]
\item \label{homogeneous}
The Green's function $\green(\Graph, \cdot\,;\divD,\divE)$ is \emph{homogeneous} of weight one; that is, the equality 
\[
\green(\Graph,a\, \mu;\divD,\divE)=a \, \green(\Graph,\mu;\divD,\divE)
\]
holds for all $a \in \rr_{\geq0}$ and for all $\mu \in \rr_{\geq0}^{\mathrm{Ed}(\Gamma)}$. 

\item \label{concave}
In the case $\divD=\divE$, the Green's function is \emph{concave}.
 Given homogeneity, this amounts to the inequality
\[
 \green\left(\Graph,\sum_{i=1}^n \mu_i;\divD,\divD\right) \geq \sum_{i=1}^n \green(\Graph,\mu_i;\divD,\divD)
\]
 for all $\mu_1,\ldots,\mu_n \in \rr_{\geq 0}^{\edges(\Graph)}$.

\item \label{monotonic}
 The Green's function is also \emph{monotonic} in the resistances: let $\resist,\resist'\in\rr_{\geq 0}^{\edges(\Graph)}$ be two resistance functions with $\resist(e)\leq\resist'(e)$ for all $e\in\edges(\Graph)$. Then 
 \[
 \green(\Graph, \resist; \divD,\divD)\leq \green(\Graph, \resist'; \divD,\divD).
 \]
 If equality holds and $(\Graph,\resist')$ is proper, then for each edge $e:i\to j$ in $\Graph$, either $\resist(e)=\resist'(e)$ or no current flows along edge $e$ when current assignment $\divD$ is induced on $(\Graph,\resist')$.
\end{enumerate}
\end{prop}

Our final result is a bound on how nonlinear the Green's function can be in the edge resistances, which will be useful for proving Theorem \ref{mainquasisplit}(b).   
 We introduce some norms for resistance functions and current assignments:
 \begin{itemize}
 \item Given $\resist \in \rr_{\geq 0}^{\edges(\Graph)}$, we let $|\resist|_1 = \sum_{e \in \edges(\Graph)} \resist(e)$.  
 \item For any zero-sum vector $\divD\in \rr^{\vertices(\Graph)}$, write $\|\divD\|$ for $\sum_{i\in \vertices(\Graph)} \max\{0, \divD_i\}$.
 If we think of $\divD$ as a current assignment, then $\|\divD\|$ is the total amount of current flowing into (and therefore out of) the network.
 \end{itemize}
 
\begin{prop} \label{bound} Let $\Graph$ be a connected graph with $\divD$ and $\divE$ two zero-sum vectors in $\rr^{\vertices(\Graph)}$.  
Then for all 
$\resist_1,\ldots,\resist_n \in \rr_{\geq 0}^{\mathrm{Ed}(\Gamma)}$
we have
\[ \left|\green\left(\Graph,\sum_{i=1}^n \resist_i;\divD,\divE\right)-\sum_{i=1}^n \green(\Graph,\resist_i;\divD,\divE)\right| \leq \|\divD\|\|\divE\| \min_{i\in\{1,\ldots,n\}} \sum_{j \neq i} |\mu_j|_1. \]
\end{prop}

\section{Labelled graphs} \label{labelled}

The purpose of this section is to recall the notion of a labelled dual graph, and to state our key formula for the height jump divisor, Theorem \ref{formulahtjump}.

Suppose for a moment that $S$ is an integral, noetherian, regular and separated scheme, and $p \colon C \to S$ is a generically smooth semistable curve. Let $U \subset S$ be the largest open subscheme such that the restriction $C_U \to U$ of $C$ over $U$ is smooth and let $D, E$ be two relative divisors of relative degree zero on $C \to S$, whose support is contained in the smooth locus $\mathrm{Sm}(C/S)$ of $C \to S$. Let $f \colon T \to S$ with $T$ an integral, noetherian, regular and separated scheme be a non-degenerate morphism. Building further upon \cite{hdj} we will express both pairings $\pair{f^*D,f^*E}_\a$ and $f^*\pair{D,E}_\a$ in terms of the geometry of the fibers of $C \to S$. Assume that the morphism $C \to S$ is \emph{quasisplit} semistable. Then at each $s \in S$, the dual graph $\Gamma_s$ (we take the definition from \cite[10.3.17]{liu}) of the fiber of $C \to S$ at $s$ is well-defined. Furthermore, the combinatorics of the singularities of the fibers is captured by the notion of  \emph{labelled} dual graph, due to second author. We will describe the admissible pairings in terms of these labelled graphs, whose definition we will now recall. We will temporarily work in slightly greater generality than in this paragraph. 

\begin{definition}
Let $\Gamma$ be a graph with finite set of edges $\mathrm{Ed}(\Gamma)$ and finite set of vertices $\mathrm{Vert}(\Gamma)$. Let $M$ be a monoid. Then an \emph{$M$-labelling} of $\Gamma$ is to be any map $\ell \colon \mathrm{Ed}(\Gamma) \to M$. Let $(\Gamma, \ell)$ be an $M$-labelled graph. A morphism $q \colon M \to N$ of monoids yields an $N$-labelled graph $(\Gamma, q\ell)$ with labelling $\mathrm{Ed}(\Gamma) \to N$ given by the composite $c \mapsto q(\ell(c))$ for any edge $c$ of $\Gamma$. For example, when the monoid of values is the additive monoid $\rr_{\geq 0}$, we reobtain the notion of a resistive network as discussed in Section \ref{resistive}.
\end{definition}

In this section we follow \cite{ho2}, in particular Remark 4.2. Let $p \colon C \to S$ be a quasisplit semistable curve over a locally noetherian scheme and $s \in S$ a point. To $s \in S$ we associate a canonical labelled graph $(\Gamma_s,\ell_s)$. The underlying graph is the dual graph $\Gamma_s$ of $C$ at $s$; it has a vertex for each irreducible component of $C_s$ and an edge for each singular point, the edge running between the vertices corresponding to components on which it lies (cf. \cite[10.3.17]{liu}). The labels take values in the multiplicative monoid $\mathrm{Princ}(O_{S,s})$ of principal ideals of the (Zariski) local ring $O_{S,s}$ of $S$ at $s$. Note that, since $S$ is integral, $\mathrm{Princ}(O_{S,s})$ coincides with the quotient $O_{S,s}/(O_{S,s})^{\times}$. The construction is as follows: let $c$ be an edge of the dual graph $\Gamma_s$ of the fiber $C_s$ of $C \to S$ at $s$, corresponding to a singular point $c \in C_s$. Then we define the label $\ell_s(c)\coloneqq(\alpha)$ for $\alpha\in \ca{O}_{S,s}$ such that the completed local ring $\widehat{O}_{C,c}$ of $C$ at $c$ is isomorphic as an $\widehat{O}_{S,s}$-algebra to $\widehat{O}_{S,s}[[x,y]]/(xy-\alpha)$. If $C \to S$ is assumed to be generically smooth then the ideal $(\alpha)$ of $O_{S,s}$ is not the zero ideal. 
Note that it is never the unit ideal. In particular, if $S$ is Dedekind, the labeled graph corresponds naturally to a metrised graph. 

\begin{example}\label{running_example_1}
Let $S = \on{Spec}\bb{C}[[u,v]]$, and let $C\to S$ be the curve in weighted projective space $\mathbb{P}_S(1,1,2)$ cut out by the affine equation
\begin{equation*}
y^2 = \bigl((x-1)^2-u\bigr)\bigl((x+1)^2-v\bigr). 
\end{equation*}
Then $C \to S$ is a quasisplit semistable curve, and is smooth over the dense open subscheme $U = D(uv) \subset S$. The labelled graph over the generic point of $S$ is a single vertex with no edges, and the labelled graph over the closed point of $S$ is a 2-gon, with edges labelled $(u)$ and $(v)$. The graph over the generic point of the closed subscheme $u=0$ (resp.\ $v=0$) is a 1-gon with label $(u)$ (resp.\ $(v)$). 
\end{example}
Canonical labelled graphs behave well with respect to pullback and specialization.

\begin{prop} \label{fiberproduct} Let $T$ be an integral noetherian scheme and let $f \colon T \to S$ be any morphism. Let $t$ be a point of $T$ and set $s =f(t) \in S$. Let $f^\# \colon O_{S,s} \to O_{T,t} $ be the induced local homomorphism. Then the labelled dual graph $(\Gamma_t,\ell_t)$ at $t$ of the base change $C \times_S T \to T$ has underlying graph $\Gamma_t=\Gamma_s$, and the labelling is given by $\ell_t=f^\#\ell_s$.
\end{prop}
\begin{proof} This is almost immediate from the definition (see \cite[Remark 2.11]{ho1}).
\end{proof} 
\begin{prop} \label{specialization} 
Assume $s, t$ are points of $S$ such that $t$ specializes to $s$, i.e.\ $s \in \overline{ \{ t \} }$. Let $\mathrm{sp} \colon O_{S,s} \hookrightarrow O_{S,t}$ be the canonical (injective) map. Then the canonical labelled graph $(\Gamma_t,\ell_t)$ at $t$ can be obtained from the canonical labelled graph $(\Gamma_s,\ell_s)$ by endowing each edge $c$ of $\Gamma_s$ with the label $\mathrm{sp}( \ell_s(c) ) \in \mathrm{Princ}(O_{T,t})$, and contracting those edges $c$ whose new label $\ell_t(c) = \mathrm{sp}(\ell_{s}(c))$ is the unit ideal of $O_{T,t}$.
\end{prop}
\begin{proof} See \cite[Section 5]{ho2}.
\end{proof}

\begin{example}\label{running_example_2}
Continuing Example \ref{running_example_1}, we find that the specialisation map from the graph over the closed point of $S$ to the graph over the generic point of $u=0$ simply contracts the edge labelled $(v)$. 
\end{example}

At this point we return to the setting from the introduction to this section, in particular $C/S$ is generically smooth. Fix a point $s \in S$. Let $(T,t)$ be a trait, with $t$ the closed point of $T$, and let $f \colon T \to S$ be a non-degenerate morphism with $f(t)=s$. Let $O_{T,t}$ be the local ring of $T$ at $t$, and let $\ord_t \colon O_{T,t} \to \zz_{\geq 0}\cup \{\infty\}$ be the normalized discrete valuation associated to $T$. Applying Proposition \ref{fiberproduct}, pulling back along $f$ gives a natural morphism of monoids $\ord_t f^\# \colon \mathrm{Princ}(O_{S,s}) \to \mathrm{Princ}(O_{T,t}) \to \zz_{\geq 0}\cup \{\infty\}$. Let $(\Gamma_s,\ell_s \colon \mathrm{Ed}(\Gamma_s) \to\mathrm{Princ}(O_{S,s}))$ be the canonical labelled graph associated to $C \to S$ at $s$. Since $f$ is non-degenerate we obtain from $f$ a $\zz_{\geq 0}$-labelled graph $(\Gamma_s,\ord_t f^\# \ell_s)$. Actually the labelling $\ord_t f^\# \ell_s$ takes values in $\zz_{>0}$ as $f^\#$ is a local homomorphism. 

We can now write down our formula for the height jump. Assume $z_i$ is a local equation in $O_{S,s}$ for the irreducible component $Z_i$ of the boundary divisor $Z=S \setminus U$. Since $\ca{O}_{S,s}$ is a regular local ring (hence a UFD), for each edge $c$ of $\Gamma_s$, the label $\ell_s(c)$ of $c$ can be written as $(z_1^{a_1}\cdots z_r^{a_r})$ for some uniquely determined $(a_1,\ldots,a_r)\in \zz_{\geq 0}^r$. 
\begin{definition}
For each $i=1,\ldots, r$ we define $\ell_{s,i}$ to be the $\mathrm{Princ}(O_{S,s})$-labelled graph obtained from $(\Gamma_s,\ell_s)$ by replacing the label $(z_1^{a_1}\cdots z_r^{a_r})$ of the edge $c$ by the principal ideal $(z_i^{a_i})$ of $O_{S,s}$. As before, bringing $f$ into the game we obtain a $\zz_{\geq 0}$-labelled graph $(\Gamma_s,\ord_t f^\# \ell_{s,i})$ from $(\Gamma_s,\ell_{s,i})$. Note that in this case, some of the labels can actually be zero, i.e.\ we have a potentially improper resistive network.
\end{definition}

Let $g(\Gamma_s,\ord_t f^\# \ell_s)$ resp.\ $g(\Gamma_s,\ord_t f^\# \ell_{s,i})$ be the Green's function of the $\zz_{\geq 0}$-labelled graphs $(\Gamma_s,\ord_t f^\# \ell_s)$ resp.\ $(\Gamma_s,\ord_t f^\# \ell_{s,i})$, using Proposition~\ref{continuous} to define the Green's function in case of an improper network. 

\begin{definition}\label{def:associated_combinatorial_divisor}
Suppose $D$ is a relative divisor on $C/S$ having support in the smooth locus $\mathrm{Sm}(C/S)$ of $C \to S$. We define a divisor $\dd$ on the dual graph of $C_s$ (i.e. $\dd \in \qq^{\mathrm{Vert}(\Gamma_s)}$) by setting, if $Y$ is an irreducible component of $C_s$, the value of $\dd(Y)$ to be the degree of the pullback of $D$ to $Y$. We call $\dd$ the `combinatorial divisor associated to $D$'. 
\end{definition}
The condition that $D$ have support in the smooth locus implies that the degrees of $D$ and $\dd$ coincide. In general we will use calligraphic font for the combinatorial divisor associated to a divisor.


Our formula for the height jump is then as follows.
\begin{thm} \label{formulahtjump} Let $S$ be an integral separated regular noetherian scheme, and let $p \colon C \to S$ be a generically smooth quasisplit semistable curve. Let $s \in S$ be a point. Let $(T,t)$ be a trait and let $f \colon T \to S$ be a non-degenerate morphism with $f(t)=s$. 
let $D, E$ be two relative divisors of relative degree zero on $C \to S$, whose support is contained in the smooth locus $\mathrm{Sm}(C/S)$ of $C \to S$.
Let $\dd, \ee$ be the combinatorial divisors associated to $D$ and $E$. Let $J(f;D,E)$ be the height jump divisor on $T$ associated to $D, E$ and $f$. Then the formula
\[ \ord_t J(f;D,E) = g(\Gamma_s,\ord_t f^\# \ell_s;\dd,\ee) - \sum_{i=1}^r g(\Gamma_s,\ord_t f^\# \ell_{s,i};\dd,\ee) \]
holds.
\end{thm}

\section{Computing the admissible pairing, and the proof of Theorem \ref{formulahtjump}} \label{basedim1}\label{explicit}

In this section we will prove Theorem \ref{formulahtjump}. We begin by describing the admissible pairing more precisely in the case where the base $S$ has dimension one. We will then treat the general case, from which the theorem will follow. 

Let $p \colon C \to S$ be a quasisplit generically smooth semistable curve over an integral separated regular noetherian scheme, and $D, E$ be relative degree zero divisors on $C$. We continue to assume that the support of both $D$ and $E$ is contained in the smooth locus $\mathrm{Sm}(C/S)$ of $C \to S$. As before (cf.\ Proposition \ref{delignepoincare}) we will freely make use of the notion of the Deligne pairing $\pair{D,E}$ as introduced in Sections 6 and 7 of \cite{de}. We leave it to the reader to verify that the Deligne pairing on relative degree zero divisors extends $\bb{Q}$-bilinearly to relative degree zero $\qq$-divisors, yielding $\qq$-line bundles on $S$. Also we recall that the Deligne pairing is compatible with arbitrary base change.
\begin{prop} \label{moretbailly}
Suppose that $C$ is regular and $S$ is a Dedekind scheme. Let $\phi_D$ be a vertical $\qq$-Cartier divisor on $C$ such that $D+\phi_D$ has zero intersection with each irreducible component of each fiber of $C \to S$. Choose $\phi_E$ in a similar way. We then have canonical isomorphisms
\[ \pair{D,E}_\a \isom \pair{D+\phi_D,E+\phi_E} \isom \pair{D,E} \otimes \pair{\phi_D,\phi_E}^{\otimes -1} \]
of $\qq$-line bundles on $S$.
\end{prop}
\begin{proof} The first follows from the proof of \cite[Th\'eor\`eme 6.15]{mb}. The second follows since $\pair{\phi_D,E+\phi_E}$ and $\pair{D+\phi_D,\phi_E}$ are canonically trivial.
\end{proof}

\begin{prop} \label{neronpairingregularC} \label{neronpairingnonregularC}
Let $S$ be a local Dedekind scheme and let $(\Gamma,\ell)$ be the canonical labelled graph associated to $C \to S$ at the closed point $s$ of $S$.
Then we have an isomorphism
\[ \pair{D,E}_\a \isom \pair{D,E} \otimes O_S(g(\Gamma,\ord_s \ell;\dd,\ee) \cdot  [s]) \]
of $\qq$-line bundles on $S$ (recall that $\dd$ is the combinatorial divisor associated to $D$ and similarly for $\ee$, cf. Definition \ref{def:associated_combinatorial_divisor}). 
\end{prop}
\begin{proof}We prove this result in two steps. We first consider the special case when $C$ is regular, and we then deduce the general case from this. 
\begin{enumerate}
\item
Assume $C$ is regular. Then every edge of $\Gamma$ has label $1$; we write $\mathbbm{1}$ for this edge labelling. We now essentially follow the arguments leading to \cite[Corollary 7.5]{hdj}. Let $F$ be the intersection matrix of the fiber of $C \to S$ at $s$, and let $L$ be the Laplacian matrix of $(\Gamma,\mathbbm{1})$. Then one easily verifies that $L=-F$. Let $\dd, \ee$ be the specializations of $D, E$ onto $\Gamma$. When viewing both $\dd$ and $\phi_D$ as elements of $\qq^{\mathrm{Vert}(\Gamma)}$ we have the matrix equation $L \cdot \phi_D = -\dd$. Letting $L^+$ be the pseudo-inverse of $L$ (see Section \ref{resistive}), we see that $\phi_D = -L^+\dd$ is a solution of the equation. Likewise we can set $\phi_E=-L^+\ee$. Now let $g(\Gamma,\mathbbm{1}; \cdot,\cdot)$ be the Green's function attached to $(\Gamma,\mathbbm{1})$, viewed as a bilinear form on the vector space of degree-zero divisors on $\Gamma$. Let $\pair{\phi_D,\phi_E}_s$ denote the local intersection multiplicity of $\phi_D$ and $\phi_E$ above $s$. We obtain that
\[ -\pair{\phi_D,\phi_E}_s = -{}^t\phi_D F \phi_E = {}^t \phi_D L \phi_E = {}^t\dd L^+LL^+ \ee = {}^t\dd L^+\ee = g(\Gamma,\mathbbm{1};\dd,\ee) \, , \]
and then the required isomorphism follows from Proposition \ref{moretbailly}. 
\item We now stop assuming that $C$ is regular. Let $C' \to C$ be the minimal desingularization of $C$ over $S$. Let $c$ be a singular point in the special fiber of $C \to S$ and assume it has label $\ell(c)=(\pi^e)$ where $\pi$ is a  uniformizer of $O_{S,s}$. Then $e \in \zz_{>0}$ is the `thickness' of the singular point, cf.\ \cite[Definition 10.3.23]{liu}, and the fiber of $C' \to C$ above $c$ consists of a chain of $e-1$ projective lines. The dual graph $\Gamma'$ of $C'$ at $s$ is hence obtained from $\Gamma$ by replacing each edge $c$ of $\Gamma$ by a chain of $\ord_s \ell(c)$ edges. It follows that $g(\Gamma,\ord_s \ell) = g(\Gamma',\mathbbm{1})$. 
We are done by step (1) once we have established a canonical isomorphism $\pair{D,E}_{C/S} \isom \pair{D,E}_{C'/S}$. But we have such a canonical isomorphism since $D, E$ do not meet the exceptional divisor of $C' \to C$.
\end{enumerate}
\end{proof}
When $S$ is Dedekind and $C$ is regular, Proposition \ref{neronpairingregularC} shows that $\pair{D,E}_\a$ coincides with the admissible pairing introduced by S. Zhang in \cite{zh}.

Now let $U \subset S $ be the largest open subscheme where $p$ is smooth, and let $Z_i$ for $ i=1,\ldots,r$ be the irreducible components of the boundary  divisor $Z=S\setminus U$. Let $V \supset U$ be the open dense subscheme of $S$ furnished by part (b) of Theorem \ref{existenceV}, and related to $D$ and $E$. 

For each $i=1,\ldots,r$ let $O_{S,z_i}$ be the local ring of $S$ at the generic point $z_i$ of the prime divisor $Z_i$. Note that $O_{S,z_i}$ is a discrete valuation ring. Let $\ord_{z_i}$ denote the normalized discrete valuation associated to $O_{S,z_i}$. Let $(\Gamma_{z_i},\ell_{z_i})$ be the canonical labelled graph of $C \to S$ at $z_i$. We can now generalise Proposition \ref{neronpairingnonregularC} to the case where $S$ is of any dimension:
\begin{prop} \label{learextension}
We have an isomorphism
\[ \pair{D,E}_\a \isom \pair{D,E} \otimes O_S\left(\sum_{i=1}^r g(\Gamma_{z_i},\ord_{z_i} \ell_{z_i};\dd,\ee)\cdot Z_i\right) \]
of $\qq$-line bundles on $S$.
\end{prop}
\begin{proof}
Put $T_i = \Spec O_{S,z_i}$ and let $f_i \colon T_i \to S$ be the canonical map. Note that $f_i$ is non-degenerate. Defining $\beta_i$ to be rational numbers such that
\[ \pair{D,E}_\a \isom \pair{D,E} \otimes O_S\left(\sum_{i=1}^r \beta_i \cdot Z_i\right) \]
as $\qq$-line bundles on $S$ we find, by pulling back along $f_i$ and using that the Deligne pairing commutes with any base change, that
\[  f_i^*\pair{D,E}_\a \isom \pair{f_i^* D,f_i^*E} \otimes O_{T_i}(\beta_i \cdot [z_i])  \]
for each $i=1,\ldots,r$.
On the other hand since localisations are flat the formation of the admissible pairing is compatible with base change along $f_i$ (cf.\ Proposition \ref{mostlytrivial}). So we find
\[  f_i^*\pair{D,E}_\a \isom \pair{f_i^*D,f_i^*E}_\a \isom \pair{f_i^*D,f_i^*E} \otimes O_{T_i}\big(g(\Gamma_{z_i},\ord_{z_i} \ell_{z_i};\dd,\ee) \cdot[z_i]\big) \]
where for the latter isomorphism we invoke Proposition \ref{neronpairingnonregularC}. The equality $\beta_i = g(\Gamma_{z_i},\ord_{z_i} \ell_{z_i};\dd,\ee)$ follows for each $i=1,\ldots,r$.
\end{proof}
\begin{example}\label{running_example_3}
Continuing Example \ref{running_example_2}, we let $Z_1 : u=0$ and $Z_2: v=0$. Then $O_{S, z_i}$ is the local ring of the generic point of $Z_i$, and we find that the graph $(\Gamma_{z_i}, \ord_{z_i} \ell_{z_i})$ is just a 1-gon with label 1. 
\end{example}
Now let $(T,t)$ be a trait and let $f \colon T \to S$ be a non-degenerate morphism. Put $s=f(t) \in S$. Define non-negative integers $m_i$ via $m_i=\ord_t f^*Z_i$ for each $i=1,\ldots,r$. Recall that the labelled graph $(\Gamma_s,\ell_{s,i})$ is obtained from the labelled graph $(\Gamma_s,\ell_s)$ by replacing any label of the form $(z_1^{a_1}\cdots z_r^{a_r})$ by the principal ideal $(z_i^{a_i})$ of $O_{S,s}$. Using this, we can compute the constants appearing in the statement of Proposition \ref{learextension}:
\begin{prop} \label{genericpoints}
For each $i=1,\ldots,r$ the equality
\[ m_i \, g(\Gamma_{z_i},\ord_{z_i} \ell_{z_i};\dd,\ee) = g(\Gamma_s,\ord_t f^\# \ell_{s,i};\dd,\ee) \]
holds.  
\end{prop}
\begin{proof} Assume first that $s \in Z_i$. Let $\mathrm{sp} \colon O_{S,s} \hookrightarrow O_{S,z_i}$ be the canonical injective morphism. From Proposition \ref{specialization} we obtain that the canonical labelled graph $(\Gamma_{z_i},\ell_{z_i})$ associated to $C \to S$ at $z_i$ is precisely the labelled graph obtained from $(\Gamma_s,\mathrm{sp}\, \ell_{s,i})$ by contracting the edges labelled with the unit ideal of $O_{S,z_i}$. In particular, the resistive network $(\Gamma_{z_i},m_i \ord_{z_i} \ell_{z_i})$ is identified with the resistive network $(\Gamma_s,\ord_t f^\# \ell_{s,i})$ with all the edges with label zero contracted. By Equation \ref{limiting-value} we find
\[ g(\Gamma_{z_i}, m_i \ord_{z_i} \ell_{z_i};\dd,\ee) = 
g(\Gamma_s,\ord_t f^\# \ell_{s,i};\dd,\ee) \, . \]
The proposition then follows by homogeneity of the Green's function (cf.\ Proposition \ref{basicgreenfacts}(\ref{homogeneous})). If $s \notin Z_i$, the labelling $\ell_{s,i}$ is identically equal to the unit ideal of $O_{S,s}$ and hence the Green's function value on the right hand side of the equation vanishes. As $m_i=0$, also the left hand side of the equation vanishes.
\end{proof}
Combining these ingredients we can finally give the proof of Theorem \ref{formulahtjump}. 
\begin{proof}[Proof of Theorem \ref{formulahtjump}] 
By applying Proposition \ref{neronpairingnonregularC} to the pullback of $p$ along $f$ we find that
\[ \pair{f^*D,f^*E}_\a \isom \pair{f^*D,f^*E} \otimes O_T\big( g(\Gamma_s,\ord_t f^\# \ell_s; \dd,\ee)\cdot[t]\big) \]
as $\qq$-line bundles on $T$. On the other hand by Proposition \ref{learextension} we have
\[ f^*\pair{D,E}_\a \isom \pair{f^*D,f^*E} \otimes O_T\left(\sum_{i=1}^r m_i g(\Gamma_{z_i},\ord_{z_i} \ell_{z_i};\dd,\ee)\cdot[t]\right) \, . \]
By Proposition  \ref{genericpoints} we therefore find
\[ f^*\pair{D,E}_\a \isom \pair{f^*D,f^*E} \otimes O_T\left(\sum_{i=1}^r   g(\Gamma_s,\ord_t f^\# \ell_{s,i};\dd,\ee )\cdot[t]\right) \, . \]
Recall that the height jump divisor on $T$ is given via an isomorphism 
\[ O_T(J(f;D,E)) \isom f^*\pair{D,E}_\a^{-1} \otimes \pair{f^*D,f^*E}_\a \, . \]
We obtain
\[ \ord_t J(f;D,E) = g(\Gamma_s,\ord_t f^\# \ell_s;\dd,\ee) - \sum_{i=1}^r g(\Gamma_s,\ord_t f^\# \ell_{s,i};\dd,\ee) \]
as required.
\end{proof}
\begin{example}\label{running_example_4}
Continuing Example \ref{running_example_3}, let $s$ denote the closed point
of $S = \Spec \cc[[u,v]]$, and fix two integers $m$, $n>0$. Let $T =
\on{Spec}\bb{C}[[t]]$, and define a map $f\colon T \ra S$ by sending $u$ to
$t^m$ and $v$ to $t^n$. If $c_u$ is the edge of $\Gamma_s$ labelled by the
ideal $(u)$ (i.e.\ $\ell_s(c_u)=(u)$), and $c_v$ the other edge, then we find 
 \begin{equation*}
 \ord_t f^\# \ell_s(c_u) = m \;\; \; \text{ and } \;\; \ord_t f^\#
\ell_s(c_v) = n \, . 
 \end{equation*} 
 Considering now the $\ord_t f^\# \ell_{s,i} $, we find 
\begin{equation*}
\ord_t f^\# \ell_{s,1} (c_u) = m \;\;\; \text{ and } \;\; \ord_t f^\#
\ell_{s,1} (c_v) = 0 \, . 
\end{equation*}
Similarly, 
\begin{equation*}
\ord_t f^\# \ell_{s,2} (c_u) = 0 \;\;\;\text{ and }\;\; \ord_t f^\#
\ell_{s,2} (c_v) = n \, . 
\end{equation*}
Suppose now that we have two sections $P$ and $O$ through the smooth locus
of $C/S$, with $P$ and $O$ specialising to different irreducible components
$\pp$ and $\oo$ of the closed fibre (we can also think of $\pp$ and $\oo$ as
the vertices of $\Gamma_s$). Let both $D, E$ be the divisor $P-O$ on $C \to S$.
We will now compute the height jump divisor $J(f;D,E)$ associated to $f$, $D$ and $E$. The divisors $\dd$, $\ee$ on the graph
$\Gamma_s$ corresponding to $D$ and $E$ are both equal to $\pp - \oo$. Applying Theorem \ref{formulahtjump} we
have the formula
\[\ord_t J(f;D,E) = g(\Gamma_s,\ord_t f^\# \ell_s;\dd,\ee) - \sum_{i=1}^2
g(\Gamma_s,\ord_t f^\# \ell_{s,i};\dd,\ee) \, . \]
Denoting effective resistance by $r_\mathrm{eff}$ we find
\begin{equation*}
g(\Gamma_s,\ord_t f^\# \ell_{s,1};\dd,\ee) =
r_{\mathrm{eff}}(\Gamma_s,\ord_t f^\#\ell_{s,1};\pp,\oo) = 0
\end{equation*}
(since the graph is a 2-gon and one of the edges has resistance zero), and
similarly 
\begin{equation*}
g(\Gamma_s,\ord_t f^\# \ell_{s,2};\dd,\ee) =
r_{\mathrm{eff}}(\Gamma_s,\ord_t f^\#\ell_{s,2};\pp,\oo) = 0 \, . 
\end{equation*}
Furthermore we compute
\[ \begin{split}
 g(\Gamma_s,\ord_t f^\# \ell;\pp-\oo,\pp-\oo) = 
r_{\mathrm{eff}}(\Gamma_s,\ord_t f^\#\ell,\pp,\oo) = \frac{mn}{m+n} \, , 
\end{split} \]
from the fact that the graph is a 2-gon with one edge labelled by $m$ and the other labelled by $n$. Putting this all together we find the non-trivial height jump
\[ \ord_t J(f;P-O,P-O) = \frac{mn}{m+n} \, .\]
\end{example}

\section{Proof of Theorems \ref{maineffective} and \ref{mainquasisplit}} \label{proofmain}

In this section we deduce Theorems \ref{maineffective} and \ref{mainquasisplit} from Theorem \ref{formulahtjump}. Again, various results on Green's functions from Propsition \ref{basicgreenfacts} (proven in the appendix) will play a crucial role.
\begin{proof}[Proof of Theorem \ref{maineffective}]
By the discussion in Definition \ref{def:semistable_curve} there exists a surjective \'etale morphism $S' \to S$ such that the semistable curve $C \times_S S' \to S'$ is quasisplit. Let $f \colon (T,t) \to S$ be a non-degenerate morphism with $(T,t)$ a trait. By Proposition \ref{mostlytrivial} the formation of $\pair{D,D}_\a$ is compatible with pullback along the \'etale morphism $S' \to S$. Hence, in order to prove Theorem \ref{maineffective} we may assume that $C \to S$ is quasisplit. Then we use the formula for the height jump from Theorem \ref{formulahtjump}. The effectivity of the height jump divisor then follows from the concavity inequality in Proposition \ref{basicgreenfacts}(\ref{concave}). 
\end{proof}
\begin{proof}[Proof of Theorem \ref{mainquasisplit}]
Let $S$ be an integral, noetherian regular separated scheme and $p \colon C \to S$ a generically smooth semistable curve. As in the theorem we assume that $p \colon C \to S$ is quasisplit, and that we are given two divisors $D, E$ of relative degree zero on $C \to S$ with support contained in $\mathrm{Sm}(C/S)$. Also we fix a point $s \in S$. Let $f \colon (T,t) \to (S,s)$ be a non-degenerate morphism with $(T,t)$ a trait and put $m_i = \ord_t f^*Z_i$ for $i=1,\ldots,r$. Let $c$ be an edge of $\Gamma_s$. Note that if $\ell_s(c)=(z_1^{a_{1c}}\cdots z_r^{a_{rc}})$ then $\ord_t f^\# \ell_s(c) = a_{1c}m_1+\cdots+a_{rc}m_r$ and $\ord_t f^\# \ell_{s,i} (c) = a_{ic}m_i$. By Theorem \ref{formulahtjump} we have
\[\ord_t J(f;D,E) = g(\Gamma_s,\ord_t f^\# \ell_s;\dd,\ee) - \sum_{i=1}^r g(\Gamma_s,\ord_t f^\# \ell_{s,i};\dd,\ee) \, . \]
As the Green's function of a resistive network is homogeneous of weight one in the labelling by Proposition \ref{basicgreenfacts}(\ref{homogeneous}), and as the labels $\ord_t f^\# \ell_s$ and $\ord_t f^\# \ell_{s,i}$ are linear forms in $m_1,\ldots, m_r$, we find the first statement in part (a). In the special case where $m_j =0$ for $j \neq i$ we get $\ell_s = \ell_{s,i}$ and the second statement in part (a) follows as well.

Finally, Proposition \ref{bound} yields a constant $c'$ depending only on $\Gamma_s$, $\dd$ and $\ee$ together with an inequality
\[ \left|\ord_t J(f;D,E)\right| \leq c' \min_{i=1,\ldots,r} \left(\sum_{j \neq i} m_j|a_j|_1 \right) \, . \]
Here we write $a_j=\sum_{c \in \mathrm{Ed}(\Gamma)} a_{jc}\delta_c$ (apply Proposition \ref{bound} with $\mu_j = m_j a_j$ for $j=1,\ldots,r$). We find that the bound in (b) holds with $$c=c'(r-1)\max_{i=1,\ldots,r; c \in \mathrm{Ed}(\Gamma)} a_{ic}. $$ 
\end{proof}
\begin{remark} Note that from Proposition \ref{bound} we actually get the effective constant $c' = \|\dd\| \|\ee\|$. It follows that the constant $c$ is also effective. 
\end{remark}

\section{Proof of Theorem \ref{greenlocal}} \label{variation}

Our next aim is to discuss our proof of the Tate-Silverman-Green Theorem \ref{greenlocal}. The key is to use our bounds on the height jump divisor from Theorem \ref{mainquasisplit}(b). Actually we would like to focus on the following more general statement.


\begin{thm} \label{generalgreen} Let $K$ be a number field or the function field of a curve, and $\bar{K}$ an algebraic closure. Let $S$ be a smooth projective geometrically connected curve over $K$, and let $U$ be an open dense subscheme of $S$ together with an abelian scheme $A \to U$ over $U$, a section $P \in A(U)$ and a section $Q \in A^\lor(U)$. 
Then there exists a $\qq$-line bundle $L$ on $S$ such that $ \deg_S L = \hat{\on{h}}_{\pp(A_\eta)}(P_\eta,Q_\eta) $, and such that the function $U(\bar{K}) \to \rr$ given by $u \mapsto \hat{\on{h}}_{\pp(A_u)}(P_u,Q_u)$ extends into a Weil height on $S(\bar{K})$ with respect to $L$.
\end{thm}
We obtain Theorem \ref{greenlocal} by letting $Q \in A^\lor(U)$ be the section given by the algebraically trivial line bundle $t_P^*\xi-\xi$. Indeed, let $u$ be any point (closed or generic) of $U$, then we have $2\,\hat{\on{h}}_{\xi_u}(P_u) = \hat{\on{h}}_{\pp(A_u)}(P_u,Q_u)$. It follows that Theorem \ref{greenlocal} is a special case of Theorem \ref{generalgreen}.

\subsection{Preliminaries}
We start by recalling a few more specialized results about the Poincar\'e bundle and its prolongations as a biextension. 
\begin{prop} \label{Poincarefunctorial} Let $U$ be a scheme, let $A \to U$ and $B \to U$ be two abelian schemes, and
let $f\colon A \to B$ be a morphism of abelian schemes over $U$. Let $f^\lor \colon B^\lor \to A^\lor$ be the dual of $f$. Then we have a canonical isomorphism of rigidified line bundles
\[ \gamma_f \colon (\mathrm{id} \times f^\lor)^*\pp(A) \isom (f \times \mathrm{id})^* \pp(B) \]
on $A \times_U B^\lor$. If $U$ is of finite type over $\cc$, then $\gamma_f$ is an isometry for the $C^\infty$ metrics induced from the canonical $C^\infty$ metrics on $\pp(A)(\cc)$, $\pp(B)(\cc)$.
\end{prop}
\begin{proof} Let $T \to U$ be a morphism of schemes, and let $P \in A_T(T)$, $Q \in B_T^\lor(T)$. We need to show that we have a canonical isomorphism of line bundles
\[ (P,f^\lor(Q))^*\pp(A_T) \isom (f(P),Q)^*\pp(B_T)  \]
on $T$. We view $Q$ as a line bundle on $B_T$ and $f^\lor(Q)$ as a line bundle on $A_T$. The left hand side is identified with the line bundle $P^*(f^\lor(Q))$ on $A_T$, the right hand side is identified with the line bundle  $f(P)^*(Q)$ on $T$. These are equal by the construction of the dual morphism. 

Suppose now that $U$ is of finite type over $\cc$. The given metrics are translation-invariant on the fibres, and the rigidification maps are isometries. These properties are stable under pull-backs along morphisms of abelian schemes, and moreover metrics with these properties are unique, so $\gamma_f$ is an isometry. 
\end{proof}
\begin{prop} \label{extensiongamma}
Let $S$ be a Dedekind scheme, and let $U \subset S$ be an open dense subscheme of $S$. Let $A \to U$ and $B \to U$ be two abelian schemes, and
let $f\colon A \to B$ be a morphism of abelian schemes over $U$. Let $f^\lor \colon B^\lor \to A^\lor$ be the dual of $f$. Denote by $\nm^0(-)$ the connected component of the N\'eron model over $S$ of an abelian scheme $-$ over $U$. Let $\bar{f} \colon \nm^0(A) \to \nm^0(B)$ and $\bar{f}^\lor \colon \nm^0(B^\lor) \to \nm^0(A^\lor)$ be the unique extensions of $f$ resp.\ $f^\lor$ furnished by the N\'eron mapping property and the connectedness of $\nm^0(-)$. Let $\bar{\pp}(A)$ resp.\ $\bar{\pp}(B)$ be the Poincar\'e prolongations of $\pp(A)$ resp.\ $\pp(B)$. Then the canonical isomorphism of rigidified line bundles $\gamma_f$ from Proposition \ref{Poincarefunctorial} extends uniquely into an isomorphism of rigidified line bundles
\[ \bar{\gamma}_f \colon (\mathrm{id} \times \bar{f}^\lor)^*\bar{\pp}(A) \isom (\bar{f} \times \mathrm{id})^* \bar{\pp}(B)  \]  
on $\nm^0(A) \times_S \nm^0(B^\lor)$.
\end{prop}
\begin{proof} This follows directly from the fact that the restriction functor \[ \mathrm{res} \colon \mathrm{Biext}(\nm^0(A),\nm^0(B^\lor);\mathbb{G}_m) \to \mathrm{Biext}(A,B^\lor;\mathbb{G}_m) \] is an equivalence of categories, cf.\ Proposition 2.8.2 of \cite{mb}.
\end{proof}
From now on we fix an integral Dedekind scheme $B$, either finite over $\Spec \zz$ or proper over a field. Write $K$ for the field of rational functions on $B$, and choose an algebraic closure $\bar{K}$. 
\begin{definition}
If $K$ is the function field of a curve (so that $B$ is a connected smooth projective curve over a field), and $L$ is a line bundle on $B$, we denote by $\deg_B L$ the usual degree of $L$ on $B$. If $K$ is a number field (so that $B$ is the spectrum of the ring of integers of $K$), and $L$ is a hermitian line bundle on $B$, we denote by $\deg_B L$ the Arakelov degree of $L$.
\end{definition}
\begin{prop} \label{poincare_height} 
Let $U \subset B$ be an open dense subscheme of $B$. Let $A \to U$ be an abelian scheme, with dual abelian scheme $A^\lor \to U$. Let $\nm(A) \to B$ resp.\ $\nm(A^\lor) \to B$ denote the N\'eron models of $A$ resp.\ $A^\lor$ over $B$, and $\nm^0(A)$ resp.\ $\nm^0(A^\lor)$ their fiberwise connected components. Let 
$\pp(A)$ denote the Poincar\'e bundle on $A \times_U A^\lor$ and let
$\bar{\pp}(A)$ be its unique prolongation as a biextension over $\nm^0(A) \times_B \nm^0(A^\lor)$. Let $x \in \nm^0(A)(B)$ and $y \in \nm^0(A^\lor)(B)$ be sections. Then the equality
\[ \hat{\on{h}}_{\pp(A_K)}(x,y) = \deg_B (x,y)^* \bar{\pp}(A) \]
holds. 
\end{prop}
\begin{proof} In \cite[Th\'eor\`eme 5.4]{mb} one finds the number field case of this result. The proof of the function field case is similar, see for example \cite[Section III.3]{pinc}. Note that if $K$ is a number field, the line bundle $(x,y)^* \bar{\pp}(A)$ is equipped with a canonical structure of hermitian line bundle by Proposition \ref{prop:adm_metric}.
\end{proof}
Let $S$ be an integral regular scheme of dimension one and of finite type over $\cc$. Let $p \colon C \to S$ be a semistable curve smooth over a dense open subscheme $U \subseteq S$. We recall from Definition \ref{metriconNeronpairing} that the Deligne pairing $\pair{D,E}$ has a canonical hermitian metric over $U(\mathbb{C})$. The following result will allow us, in the situation of Theorem \ref{generalgreen}, to construct a Weil height with respect to $L$ using a structure of continuous hermitian line bundle on an extension $\ca{L}$ of $L$ over a suitable surface $\ca{S}$ extending $S$.   
\begin{prop} \label{ctsmetric}
The $\qq$-line bundle $\pair{D,E}_\a$ on $S(\cc)$ can be equipped with a unique continuous hermitian metric extending the canonical hermitian metric on the restriction of $\pair{D,E}$ to $U(\cc)$.
\end{prop}
\begin{proof} See \cite[Theorem 2.2]{hdj}.
\end{proof}
The proof of Theorem \ref{generalgreen} will be broken up into two subsections. Subsection \ref{sub:nice_curve_case} is a proof of the theorem in a rather special situation involving jacobians of curves with particularly nice integral models. In subsection \ref{sub:reduction} we will then show how to reduce the general case to this special case. 

\subsection{Proof of Theorem \ref{generalgreen} for suitable jacobians}\label{sub:nice_curve_case}
Let $B$ and $K$ be as above, so $K$ is a number field or the function field of a curve. Suppose we are given: 
\begin{enumerate}
\item a proper flat scheme $\ca{S}$ of relative dimension $1$ over $B$, regular and with connected geometric fibres;
\item a quasisplit semistable curve $\ca{C} \to \ca{S}$, smooth over a dense open subscheme of $\ca{S}$ (write $\ca{U} \subset \ca{S}$ for the largest such open);
\item horizontal divisors $\ca{D}$, $\ca{E}$ on $\ca{C} \to \ca{S}$ supported on the smooth locus $\mathrm{Sm}(\ca{C}/\ca{S})$ and both of relative degree zero, 
\end{enumerate}
such that
\begin{enumerate}
\item[] (*) the $B$-horizontal irreducible components of the complement $\ca{S} \setminus \ca{U}$ are disjoint.
\end{enumerate}
Denote the base changes of $\ca{S}$, $\ca{U}$, $\ca{C}$, $\ca{D}$ resp.\ $\ca{E}$ to $K$ by $S$, $U$, $C$, $D$ resp.\ $E$. Let $J_U$ be the jacobian of $C_U$ over $S$ and denote by $\mu \colon J_U \isom J_U^\lor$ the canonical principal polarization of $J_U$. 
\begin{thm} \label{generalgreenspecial}
Under the above assumptions, the conclusion of Theorem \ref{generalgreen} holds for the jacobian scheme $J_U \to U$ and the sections $P$ resp.\ $Q$ of $J_U$ resp.\ $J^\lor_U$ corresponding to $D$ resp.\ $\mu E$.
\end{thm}
The idea of the proof of Theorem \ref{generalgreenspecial} is as follows: we find a suitable continuous hermitian  
$\qq$-line bundle $\ll$ on $\ss$, together with a real number $j$, such that for all $u
\in S(\bar{K})$ the difference between the height of $u$ with respect to
$\ll$ and the N\'eron-Tate height pairing between $D_u$ and $E_u$ on  
$J_u$ is bounded by $j$. The constant $j$ will be a bound on a sum of local height jumps indexed by a finite set of non-archimedean primes of $K$. 
This finite set of primes depends only on the data in (1), (2) and (3); a bound for the local height jump at each given prime is provided by Theorem \ref{mainquasisplit}(b).

To simplify the notation (in particular the normalisations of valuations), we will assume for the remainder of this section that $K$ is a number field. The argument in the function field case is similar. 
\begin{proof}[Proof of Theorem \ref{generalgreenspecial}]
Let $\ca{U} \subseteq \ca{V}\subseteq \ca{S}$ be an open subscheme with complement of codimension at least two, and over which some positive multiples of the sections $P$ and $Q$ extend to the identity components of the N\'eron models of $J_U$ and $J_U^\lor$ respectively, cf. part (b) of Theorem \ref{existenceV}. Write $Z=\sum_{i=1}^r Z_i$ for the decomposition of $Z=\ss \setminus \uu$ into irreducible components. We then have that $\ss \setminus \VV \subset \cup_{i \neq j} Z_i \cap Z_j$ is a finite union of closed points in $\ss$.  

Let $\ll$ be the $\qq$-line bundle $\pair{\dd,\ee}_\a^{\otimes -1}$ on $\ss$, i.e.\ the dual of the admissible pairing between $\dd$ and $\ee$ over $\ss$. The $\qq$-line bundle $\ll$ restricts to the line bundle $\pair{D,E}_\a^{\otimes -1}$ on $S=\ss_K$, and hence Proposition \ref{ctsmetric} implies that $\ll$ has a canonical structure of continuous hermitian $\qq$-line bundle on $\ss$. Write $L=\pair{D,E}_\a^{\otimes -1}$. We claim that $L$ satisfies the conclusion of the theorem.
By Proposition \ref{poincare_height} we have $\deg_S L = \hat{\on{h}}_{\pp(J_\eta)}(D_\eta, \mu E_\eta)$, so this takes care of part of the theorem already. 

To prove the statement about the variation of $\hat{\on{h}}_{\pp(J_u)}$, we define a suitable collection (ranging over all finite subextensions $K \subset K' \subset \bar{K}$ and over all places $v$ of $K'$) of local heights associated to $L$, using as input the model $\ss$ of $S$ and a non-zero rational section $\ell$ of the continuously metrized line bundle $\ll$ over $\ss$, which we fix from now on. We proceed as follows. Let $f \colon T \to \ss$ with $T$ regular, integral be a finite flat quasi-section over $B$. Write $K'$ for the function field of $T$, and $u \colon \Spec K' \to S$ for the restriction of $f$ to the generic point of $T$. Without loss of generality we may assume that the image of $u$ is disjoint from the support of the divisor of $\ell_K$ (which is a finite union of closed points of $S$).

Consider first of all a non-archimedean prime $v$ of $K'$. Then let $T_v \to T$ be the localisation of $T$ at $v$ and let $f_v \colon T_v \to \ss$ be the composite of $f$ with $T_v \to T$. We then put $\lambda_{L,v}(u) = \ord_v f_v^* \divisor_\ss \ell$. For $v \colon \Spec \cc \to T$ an archimedean prime of $K'$ we put $\lambda_{L,v}(u) = -\log \| f_v^*\ell \|$, where now $f_v$ is the composite of $f \colon T \to \ss$ with the morphism $v \colon \Spec \cc \to T$. Here the norm $\|\cdot \|$ is the norm on $\pair{D,E}_\a^{\otimes -1}$ over $U_v(\cc)$ furnished by Definition \ref{metriconNeronpairing}. We now note that 
\[ [K':\qq]h_L(u) = \sum_{v \in M_{K'}} \lambda_{L,v}(u) \log Nv \] 
(with $Nv$ the `cardinality' of the local residue fields) for $u \in U(K')$  gives, by varying $K'$ over the finite subextensions of $K \subset \bar{K}$, the height on $S(\bar{K})$ associated to the continuously metrized line bundle $\ll$ and its rational section $\ell$ on $\ss$. 

Likewise, we can decompose $\hat{\on{h}}_{\pp(J_u)}(D_u,\mu E_u)$ into local contributions at all places $v \in M_{K'}$, using Proposition \ref{poincare_height}. Thus we view $f_v^*\ell$ as a non-zero rational section of the line bundle $\pair{f_v^*\dd,f_v^*\ee}_\a^{\otimes -1}$ on $T_v$. This allows us to define a local contribution $\hat{\lambda}_{\pp(J_u),v}(D_u, \mu E_u) $ to the canonical height as the local multiplicity $\ord_v \divisor_{T_v} f_v^* \ell$, in the non-archimedean case, and by putting (exactly similar to the earlier archimedean contribution) $\hat{\lambda}_{\pp(J_u),v}(D_u, \mu E_u)   =  -\log \| f_v^*\ell \|$ in the archimedean case. We find that for each $u \in S(K')$, the identity 
\[ [K':\qq]\hat{\on{h}}_{\pp(J_u)}(D_u, \mu E_u) = \sum_{v \in M_{K'}} \hat{\lambda}_{\pp(J_u),v}(D_u, \mu E_u) \log Nv \]
holds.

We are done once we show that there exists a constant $c$ and a finite set of places $N$ of $B$ such that the difference 
\[ \left|\hat{\lambda}_{\pp(J_u),v}(D_u,\mu E_u) - \lambda_{L,v}(u) \right| \] with $u$ ranging over $S(K')$, and $v$ ranging over the primes of $K'$, and $K'$ ranging over all finite subextensions of $K \subset \bar{K}$, is non-zero only for $v$ lying above $N$, and moreover for such $v$ is bounded by $c \cdot e(v/b)$, where $b$ is the prime of $N$ lying under $v$ and $e(v/b)$ is the ramification index of $v$ over $b$. 

First of all, the required difference is zero for archimedean $v$ so we can restrict our attention to the closed points $v$ of $T$. For closed points $v$ of $T$ the required difference is precisely the coefficient of the height jump divisor $J(f_v;\dd,\ee)$ on $T_v$ associated to $f_v$, $\dd$ and $\ee$. It follows that we need only restrict our attention to those closed points $v$ whose image is contained in the finite subset $Z'=\cup_{i \neq j} Z_i \cap Z_j$ of $\ss$ since for the other closed points, the jump in the height vanishes. We take $N$ to be the image $\pi(Z')$ of the finite closed subset $Z'$ in $B$.

Let $q$ be the given map $T \to B$, let $v$ be a closed point of $T$ and write $b=q(v)$. Let $F$ be the closed fiber of $\pi \colon \ss \to B$ above $b$. As divisors on $T$ we have $f^*F = q^*b$ hence $\ord_v f^*F = \ord_v q^*b= e(v/b)$, the ramification index of $v$ over $b$. Assume that $s =f(t) \in Z'$.
By condition (*) on $\ss$, precisely two components $Z_1, Z_2$ of $Z$ pass through $s$, and at least one is vertical. Assuming $Z_1$ is vertical, we have $\ord_v f^* Z_1 =\ord_v f^*F = e(v/b)$. By part (b) of Theorem \ref{mainquasisplit} there exists a constant $c$ depending only on $\dd$, $\ee$ and the morphism $\CC \to \ss$ such that the bound $|\ord_v J(f;\dd,\ee)| \leq c \cdot e(v/b)$ holds. This finishes the proof of Theorem \ref{generalgreenspecial}. 
\end{proof}

\subsection{Reduction of Theorem \ref{generalgreen} to the case of jacobians} \label{sub:reduction}
We will now show that Theorem \ref{generalgreenspecial} implies Theorem \ref{generalgreen}. We proceed in two steps: first, a reduction from abelian schemes to jacobians, and then a reduction from jacobians to jacobians of the special kind introduced in subsection \ref{sub:nice_curve_case} and for which we know that the conclusion of Theorem \ref{generalgreen} holds. This section is essentially a long sequence of fairly standard reduction steps. We start by relating our abelian scheme $A \to U$ and the point $P \in A(U)$ to a divisor on some jacobian. 
\begin{prop} \label{surjfromjac} Let $U$ be an integral scheme with infinite field of fractions, $A \to U$ an abelian scheme, $P \in A(U)$ a section. Then, possibly after replacing $U$ by a dense open subset, there exist a smooth projective curve $C \to U$ with geometrically connected fibers, a morphism $g \colon J \to A$ over $U$, where $J=\mathrm{Pic}^0(C/U) \to U$ is the jacobian of $C \to U$, and a pair of sections $a$, $b \in C(U)$ such that, viewing $D = b-a$ as a divisor on $C$ of relative degree zero over $U$, we have $P = g([O_C(D)])$.
\end{prop}
\begin{proof} We slightly strengthen the argument given in \cite[Section 2]{ldg} so that we can lift the section $P$. We start with the following remark. Let $K$ be an infinite field with separable closure $\bar{K}$, and $l\in \zz$ a prime number unequal to the characteristic of $K$. Let $X\subset \PP^n_K$ be a smooth projective connected scheme, and $p$, $q \in X(K)$ be distinct points. Then there exists a smooth connected projective curve $C_0 \subset X$ such that the following two conditions hold:
\begin{itemize}
\item[(1)] $p, q \in C_0(K)$;
\item[(2)] the canonical map $\mathrm{H}^1_{et}(X\times_K\bar{K}, \zz_l) \to
\mathrm{H}^1_{et}(C_0\times_K\bar{K}, \zz_l)$ is injective.
\end{itemize}
This follows from \cite[Theorem 3.1]{Diaz1991Strong-Bertini-} and the Lefschetz hyperplane theorem. Note that remark (a) on page 80 of \cite{Diaz1991Strong-Bertini-} shows that the curve constructed in \cite[Theorem 3.1]{Diaz1991Strong-Bertini-} is connected, and the remark at the end of page 81 of \cite{Diaz1991Strong-Bertini-} shows that the curve can be assumed to be defined over $K$, since $K$ is infinite.  

Turning now to the situation of the proposition, write $\eta =\Spec K$ for the generic point of $U$. Then $A_K$ is an
abelian variety over an infinite field, with two marked points, the
identity $e_K$ and the base-change $P_K$ of the section $P$. By what we said above there exists a smooth connected projective curve $C_0 \subset A_K$ such that $e_K$ and $P_K$ both lie in $C_0(K)$, and such that the canonical map $\mathrm{H}^1_{et}(A\times_K\bar{K}, \zz_l) \to
\mathrm{H}^1_{et}(C_0\times_K\bar{K}, \zz_l)$ is injective. Let $a_0$, $b_0 \in C_0(K)$ be the points corresponding to $e_K$ and $P_K$ respectively. 

Write $J_0 = \mathrm{Pic}^0_{C_0/K}$ for the Jacobian of $C_0$, and
\[
\alpha\colon  C_0  \to J_0 \, , \quad q  \mapsto [q-a_0] \\
\]
for the Abel-Jacobi map with base point $a_0$. Then $\alpha\colon  C_0 \to J_0$ is the Albanese of $(C_0, a_0)$, and so the inclusion $C_0 \to A_K$ induces a map $g\colon J_0 \to A_K$ with $\alpha(a_0) = e_K$ and $\alpha(b_0) = P_K$. Moreover, the injectivity of the map on cohomology ensures that this map on abelian varieties is surjective. 

Finally, applying arguments from \cite[\S 8]{EGAIV3}, we can ``spread out'' the curve $C_0$, the sections $a_0$ and $b_0$ and the map $g$ so that they are all defined over some open neighbourhood of $\eta$ in $U$. 
\end{proof}

\begin{cor} \label{heightcomparison}
Let $K$ be a number field or the function field of a curve. Let $S$ be a smooth projective geometrically connected curve over $K$, and let $U$ be an open dense subscheme of $S$ together with an abelian scheme $A \to U$ over $U$, a section $P \in A(U)$ and a section $Q \in A^\lor(U)$. Then, possibly after replacing $U$ by a dense open subset, there exist a smooth curve $C \to U$ with geometrically connected fibers and divisors $D, E$ of relative degree zero on $C \to U$ such that
\[ \hat{\on{h}}_{\pp(A_\eta)}(P_\eta,Q_\eta) =
\hat{\on{h}}_{\pp(J_\eta)}(D_\eta,\mu E_\eta) \, ,\quad 
\hat{\on{h}}_{\pp(A_u)}(P_u,Q_u) =
\hat{\on{h}}_{\pp(J_u)}(D_u,\mu E_u) \]
for all finite subextensions $K \subset K' \subset \bar{K}$, and all $u \in U(K')$. Here $J \to U$ is the jacobian of $C \to U$ with principal polarization $\mu \colon J \isom J^\lor$. 
\end{cor}
\begin{proof} By Proposition \ref{surjfromjac} we can assume there exists a smooth curve $C \to U$ and a morphism $g \colon J \to A$ where $J \to U$ is the jacobian of $C \to U$, and a divisor $D = b-a$ of relative degree zero on $C \to U$ such that $P=g([O_C(D)])$. Then take $E$ a divisor of relative degree zero on $X \to U$ such that $\mu[O_X(E)]=g^\lor(Q)$. To see that such a divisor $E$ exists, note that $X \to U$ has a section (e.g. $a$) so by \cite[Proposition 8.1.4]{blr} there exists a line bundle $M$ on $X \to U$ representing $\mu^{-1}g^\lor(Q)$. Note that $M$ has relative degree zero. Choosing a non-zero rational section of $M$ then gives $E$. To obtain the identities, combine Propositions \ref{extensiongamma} and \ref{poincare_height}.
\end{proof}
From Corollary \ref{heightcomparison} it follows that it suffices to prove the special case of Theorem \ref{generalgreen} dealing with jacobians.

Hence, in view of Theorem \ref{generalgreenspecial}, our final step is to reduce the case of jacobians to the case of jacobians of quasisplit semistable curves satisfying the various other hypotheses at the beginning of subsection \ref{sub:nice_curve_case}. First, we show that a semistable curve becomes quasisplit after a suitable alteration of the base.

\begin{prop}\label{lem:quasi_split_after_alteration}
Let $f\colon C \ra S$ be a semistable curve over an integral noetherian excellent scheme. Then there exists an alteration $S' \ra S$ such that the morphism $C' = C \times_S S' \ra S'$ is quasi-split. 
\end{prop}
\begin{proof}
After replacing $S$ by an alteration, we may assume that there exist
sections $\sigma_1,\ldots, \sigma_n \in C(S)$ such that 
\begin{enumerate}
\item for every geometric point $\bar{s}$ of $S$ and every connected
component $X$ of $C^{\on{sm}}_{\bar{s}}$, there exists $i$ such that
$\sigma_i(\bar{s}) \in X$ (by \cite[6.2]{alt}). 
\item
for every geometric point $\bar{s}$ of $S$ and every singular point
$\bar{x}$ of $C_{\bar{s}}$, there is an $i$ such that $\bar{x} \in
\sigma_i(\bar{s})$ (by \cite[6.3]{alt}).
\end{enumerate}
It suffices to show that $C/S$ is quasisplit. Condition (1) immediately
implies that all irreducible components of all fibres are geometrically
irreducible, as an irreducible scheme over a field which has a smooth
rational point is geometrically irreducible. It is also clear from condition
(2) that every singular point in every fibre is rational. It remains to
check that the morphism $\on{Sing}(C/S) \ra S$ is source-Zariski-locally an
immersion. 

Let $x \in \on{Sing}(C/S)$ be a point, let $s$ denote the image of $x$ in
$S$, and let $R = O_{S,s}$. We may and do assume that $S = \on{Spec}R$.
Let $X$ be the connected component of $\on{Sing}(C/S)$ which contains $x$.
Then $X \ra S$ is finite hence affine, so write $X = \on{Spec}A$, for some
$R$-algebra $A$. Note that $A$ is connected, and is finite and unramified as
an $R$-algebra. We need to show that $A$ is a quotient of $R$. 

By condition (2), we know that every irreducible component of
$\on{Sing}(C/S)$ is contained in the image of some section in $C(S)$ (in
general different sections for different components). As such, there exist
closed subschemes $Z_1, \ldots, Z_n$ of $S$ such that $X$ can be written (as
a topological space over $S$) as a union of the $Z_i$. Since $X \ra S$ is
separated, these $Z_i$ are closed subschemes of $X$, and so in particular
their intersections are also closed. Since $S$ is local, this implies that
the fibre of $X$ over the closed point of $S$ is (as a topological space) a
single point. 

Now $A$ is finite and unramified over $R$, so it follows that (writing $\mathfrak{m}$
for the maximal ideal of $R$) we have $A/\mathfrak{m}A = R/\mathfrak{m}$.  Then by Nakayama's lemma we find that $R$ surjects onto $A$ and we are done. 
\end{proof}
We are now ready to show that, after several further alterations, we can arrange the situation at the beginning of \ref{sub:nice_curve_case}. Recall that we start out with a number field $K$, a smooth projective geometrically connected curve $S/K$, a dense open $U \subset S$ and a smooth proper curve $C \to U$ with connected geometric fibers. Let $B$ be the spectrum of the ring of integers of $K$. Let $J \to U$ denote the jacobian of $C \to U$, and $J^\lor$ its dual. We assume sections $P \in J(U)$ and $Q \in J^\lor(U)$ are given, represented by the classes of $D$ resp.\ $\mu E$ (cf. Corollary \ref{heightcomparison}), where $\mu \colon J \isom J^\lor$ is the canonical principal polarization of $J \to U$ and $D$, $E$ are divisors of relative degree zero on $C \to U$. 
\begin{prop} \label{lem:alteration_to_nice}
There exist 
\begin{enumerate}
\item
a proper flat morphism $\pi\colon \ca{S} \ra B$ with connected geometric fibres;
\item a quasisplit semistable curve $p\colon \ca{C} \ra \ca{S}$ smooth over a dense open subscheme of $\ca{S}$ (we write $\ca{U}\subset \ca{S}$ for the largest such open subscheme);
\item an open immersion $f\colon V \ra \ca{S} \times_B K$;
\item a finite flat morphism $g\colon V \ra U$;
\item an isomorphism $h\colon(\ca{C}\times_B K) \times_{(\ca{S} \times_B K)}V \ra C\times_U V$;
\item two horizontal divisors $\ca{D}$, $\ca{E}$ of relative degree zero on $\ca{C}$ supported on the smooth locus of $\ca{C} \to \ca{S}$;
\end{enumerate}
satisfying
\begin{enumerate}
\item writing $Z = \ca{S} \setminus \ca{U}$, we have that $Z$ has its $B$-horizontal irreducible components disjoint;
\item the image of $\ca{D}$ (resp.\ $\ca{E}$) under $h$ is equal to the pullback of $D$ (resp.\ $E$) along $C\times_S V \ra C$. 
\end{enumerate}
\end{prop}
Note that the finite flat map $g\colon V \ra U$ automatically extends to a finite flat map $\bar{g}\colon \ca{S}\times_B K \ra S$, since the source is a
smooth proper curve. 
\begin{proof}
This proof consists of a number of standard arguments (mostly from \cite{alt}), so we omit some details. We begin by constructing proper flat models $\ca{S}_0$, $\ca{C}_0$ of $S$ and $C$ over $B$, together with divisors $\ca{D}_0$ and $\ca{E}_0$. We will then apply repeated alterations to these objects to get them into the required form. 

Over some dense open subscheme $\ca{U}_0$ of $\ca{S}_0$ we find that $\ca{D}$ and $\ca{E}$ are horizontal, and taking a finite flat cover we can even assume they can be written as formal sums taken from a finite set $\Sigma$ of sections. Perhaps shrinking $\ca{U}_0$ we can assume that $\ca{C}_0$ is smooth and proper over $\ca{U}_0$, and that the sections in $\Sigma$ are disjoint over $\ca{U}_0$, so that the pair $(\ca{C}_0|_{\ca{U}_0}, \Sigma)$ induces a moduli map $m\colon \ca{U}_0 \ra \ca{M}_{g, n}$ where $g$ is the genus of the fibres of $\ca{C}_0$ and $n$ is the cardinality of $\Sigma$. Now the compactified moduli stack $\overline{\ca{M}}_{g, n}$ admits a finite surjective generically \'etale map from a scheme (see \cite[Th\'eor\`eme 16.6]{Laumon2000Champs-algebriq}), so by taking the closure of the graph of the map $m$ we obtain an alteration $\ca{S}_1 \ra \ca{S}_0$ such that $\ca{C}_0$ admits a semistable model $\ca{C}_1$ over $\ca{S}_1$, and such that the divisors $\ca{D}_0$ and $\ca{E}_0$ pull back to divisors $\ca{D}_1$ and $\ca{E}_1$ whose supports are unions of disjoint sections through the smooth locus of $\ca{C}_1$ over $\ca{S}_1$ (this is a slight variant of one of the main arguments of \cite{ldg}). 

By Proposition \ref{lem:quasi_split_after_alteration}, after further alteration of $\ca{S}_1$ we can arrange that $\ca{C}_1/\ca{S}_1$ is quasisplit. Write $\ca{U}_1 \subset \ca{S}_1$ for the largest open subscheme over which $\ca{C}_1$ is smooth. Blowing up $\ca{S}_1$ more, we may assume that the $B$-horizontal irreducible components of $\ca{S}_1\setminus\ca{U}_1$ are disjoint. 
\end{proof}
\begin{proof}[Proof of Theorem \ref{generalgreen}]
By Theorem \ref{generalgreenspecial} and Proposition \ref{lem:alteration_to_nice} 
we are left to show that if $\bar{g} \colon S' \to S$ is an alteration, and  Theorem \ref{generalgreen} holds for the pullbacks of $C \to U \subset S$ and the sections $P$ and $Q$ along $\bar{g}$, then Theorem \ref{generalgreen} holds for $C \to U \subset S$ and the sections $P$ and $Q$ themselves. 

As above, to reduce the amount of notation, we write $\nmo(-)$ for the fibrewise connected component of the N\'eron model of an abelian variety. Then we find that $P$ and $Q$ extend (after taking some multiples, which we will suppress in the notation) to elements of $\nmo(J)(S)$. Write $\bar{\ca{P}}$ for the Poincar\'e prolongation on $\nmo(J) \times_S \nmo(J^\lor)$. Writing $\sigma$ for the section $(P, Q)$ of $\nmo(J) \times_S \nmo(J^\lor)$, we set $L = \sigma^*\bar{\ca{P}}$, a line bundle on $S$. We claim that this line bundle fulfils the conclusion of Theorem \ref{generalgreen}; we now need to prove the corresponding statements about heights. 

Write $J'$ for the jacobian of $\bar{g}^*C$ over $\bar{g}^{-1}(U)$, and $J'^\lor$ for its dual. We again have a N\'eron model component  $\nmo(J')$, and the same (suppressed) multiples of the sections $\bar{g}^*P$, $\bar{g}^*Q$ extend to a section $\tau$ of $\nmo(J') \times_{S'} \nmo(J'^\lor)$. Set $L' = \tau^*\bar{\ca{P}}'$, where $\bar{\ca{P}}'$ is the Poincar\'e prolongation on $\nmo(J') \times_{S'} \nmo(J'^\lor)$. 

We will now show $\bar{g}^*L \cong L'$. For this, note first that the N\'eron mapping property (and the fact that the image of a connected scheme is connected) yields a unique $S'$-morphism
\begin{equation*}
f\colon \bar{g}^*\left( \nmo(J) \times_S \nmo(J^\lor) \right)  \ra \nmo(\bar{g}^*J) \times_{S'} \nmo(\bar{g}^*J^\lor)  
\end{equation*}
extending the canonical isomorphism over $\bar{g}^{-1}U$. 
Writing $\sigma'$ for the pullback along $\bar{g}$ of $\sigma$, and recalling that $\tau$ denotes the extension of $(\bar{g}^*P, \bar{g}^*Q)$ to a section in $ \left(\nmo(\bar{g}^*J) \times_{S'} \nmo(\bar{g}^*J^\lor)\right)(S')$, we have $f(\sigma') = \tau$, by the uniqueness part of the N\'eron mapping property. Finally we have that $f^*\bar{\ca{P}}' = \bar{g}^*\bar{\ca{P}}$, by the uniqueness of extensions of rigidified bundles. This implies that $\bar{g}^*L = L'$. 

Let $\on{h}_L$ resp.\ $\on{h}_{L'}$ be Weil heights with respect to $L$ on $S$ resp.\ $L'$ on $S'$. Then the quantity
\begin{equation*}
\abs{\on{h}_{L}(\bar{g}(u)) - \on{h}_{L'}(u)}
\end{equation*}
is bounded uniformly as $u$ runs over $S'(\bar{K})$. Next we can assume that 
\[ \abs{\hat{\on{h}}_{\pp(J')}((\bar{g}^*P)_u,(\bar{g}^*Q)_u) - \on{h}_{L'}(u)} \]
is bounded on $S'(\bar{K})$ by Theorem \ref{generalgreenspecial}. We are then done by the triangle inequality, noting that 
\[ \hat{\on{h}}_{\pp(J')}((\bar{g}^*P)_u,(\bar{g}^*Q)_u)  = \hat{\on{h}}_{\pp(J)}(P_{\bar{g}(u)},Q_{\bar{g}(u)}) \]
for all $u \in S'(\bar{K})$. 
%
\end{proof}

\section{Proof of Theorems \ref{nefnessproperty} and \ref{nefnessproperty_Arakelov}} \label{nef}

The proofs of Theorems \ref{nefnessproperty} and \ref{nefnessproperty_Arakelov} can be given both at the same time. Note that by Theorem \ref{maineffective} the height jump divisor $J(f;D,D)$ is an effective divisor. It therefore suffices to show that $\pair{f^*D,f^*D}^{\otimes -1}_\a$ has non-negative (Arakelov) degree on $T$. But by Proposition \ref{poincare_height} the latter is precisely the canonical height, with respect to the principal polarization of the jacobian of the generic fiber of $f^*C \to T$, of the point determined by $D$. As such canonical heights are non-negative, we obtain the result.


\appendix

\section{Results on Green's functions} \label{resultsongreens}

Our goal in this appendix is to study the Green's function on a network in somewhat more detail using the techniques of resistor networks.  Specifically, we will prove that the Green's function for proper networks is homogeneous, concave, and monotonic in the manner described in Proposition \ref{basicgreenfacts}, and then we will prove that the Green's function extends continuously with these properties to improper networks, as described in Proposition \ref{continuous}.
 Finally, we prove that the nonlinear part of the Green's function is bounded in the manner described in Proposition \ref{bound}.  

For the proofs of homogeneity and continuity, we will first develop a graphical calculus for evaluating the Green's function, so we establish more notation and terminology.  
Given a resistive network $(\Graph, \resist)$, we extend $\resist$ to a real-valued function on the collection of subsets of $\edges(\Graph)$ by setting, for each subset $S\subset\edges(\Graph)$,
\[\resist(S):= \prod_{e\in S}\resist(e).\]

Recall that a \emph{subgraph} of a graph $\Graph=(V,E,\epoints)$ is a graph $(V', E', \epoints')$ where $V'$ and $E'$ are subsets of $V$ and $E$, and $\epoints'$ is the restriction of $\epoints$ to $E'$.  
We will only be considering subgraphs for which $V'=V$, and which are thus entirely determined by the set of edges they include.
We denote the subgraph $(\vertices(\Graph), S, \epoints|_{S})$ associated to a set $S\subset\edges(\Graph)$  by $\Graph|_S$.

A \emph{walk} in a graph $\Graph$ is a finite sequence of vertices of $\Graph$, together with a choice of an edge of $\Graph$ from each vertex to the next. 
A \emph{path} is a walk where all vertices are distinct, except possibly the start and end vertices. 
A \emph{cycle} is a path with at least one edge, for which the start and end vertex are equal.

Recall that a \emph{tree} is a graph which is connected and cycle-free.  
A \emph{spanning tree} for a graph $\Graph$ is a subgraph of $\Graph$ which is a tree and contains all the vertices of $\Graph$.  
We denote the collection of sets of edges $\tree\subset\edges(\Graph)$ such that $\Graph|_\tree$ is a spanning tree for $\Graph$ by $\trees(\Graph)$.
More generally, if a graph is merely cycle-free it is called a \emph{forest}.  
We denote by $\forests{n}(\Graph)$ the collection of subsets $\forest\subset\edges(\Graph)$ such that $\Graph|_\forest$ is a forest and has $n$ connected components.
Informally, we refer to elements of $\trees(\Graph)$ as spanning trees, and to elements of $\forests{n}(\Graph)$ as $n$-forests.

One more ingredient is needed for our formula for the Green's function.  Given $\forest\in\forests{2}(\Graph)$ and two pairs $(i,j)$ and $(k,\ell)$ of vertices of $\Graph$, we define $\permute{i}{j}{k}{\ell}{\forest}$ as follows:
\[\permute{i}{j}{k}{\ell}{\forest} := \begin{cases}
+1 & \text{ if one component of $\Graph|_\forest$ contains $i$ and $k$ } \\
 & \text{     and the other contains $j$ and $\ell$;}\\
-1 & \text{ if one component of $\Graph|_\forest$ contains $i$ and $\ell$} \\ 
 & \text{     and the other contains $j$ and $k$;}\\
\ 0 & \text{ if $i$ and $j$ or $k$ and $\ell$ are in the same component of $\Graph|_\forest$.}
\end{cases}\]

\begin{prop} \label{prop:tree-formula}
Let $(\Graph, \resist)$ be a proper resistive network with $\Graph$ connected, and let $i$, $j$, $k$, and $\ell$ be vertices of $\Graph$.  
Then 
\begin{equation}\label{tree-formula}
 \green(\Graph, \resist; \basis_i-\basis_j, \basis_k-\basis_\ell) = \frac{\sum_{\forest\in\forests{2}(\Graph)}\permute{i}{j}{k}{\ell}{\forest}\resist(\edges(\Graph)\setminus\forest)}{\sum_{\tree\in \trees(\Graph)}\resist(\edges(\Graph)\setminus\tree)}.
\end{equation}
\end{prop}
\begin{proof}
 This amounts to finding a voltage distribution that induces a current of $+1$ at vertex $k$, a current of $-1$ at vertex $\ell$, and a current of $0$ elsewhere, and then checking the voltage difference between vertices $i$ and $j$.
 
 Let $\mathcal{F}_{k\ell}\subset\forests{2}(\Graph)$ be the set of those $2$-forests $\forest$ such that $k$ is in one component of $\Graph|_\forest$ and $\ell$ is in the other.  
 For each vertex $i$ of $\Graph$, define a function $\indicator{i}:\mathcal{F}_{k\ell}\to\{0,1\}$ by
 \[\indicator{i}(\forest)=\begin{cases}\text{$1$ if $i$ is in the component of $\Graph|_\forest$ containing $k$}\\\text{$0$ if $i$ is in the component of $\Graph|_\forest$ containing $\ell$.}\end{cases}\]
 Then define a voltage distribution $v\in \rr^{\vertices(\Graph)}$ by 
 \[v_i := \frac{\sum_{\forest\in\mathcal{F}_{k\ell}}\indicator{i}(\forest)\resist(\edges(\Graph)\setminus\forest)}{\sum_{\tree\in \trees(\Graph)}\resist(\edges(\Graph)\setminus\tree)}.\]
 This voltage distribution is so designed that the voltage difference $v_i-v_j$ is the quantity in Equation \eqref{tree-formula}:
 \begin{align*}
  v_i-v_j &= \frac{\sum_{\forest\in\mathcal{F}_{k\ell}}(\indicator{i}(\forest)-\indicator{j}(\forest))\resist(\edges(\Graph)\setminus\forest)}{\sum_{\tree\in \trees(\Graph)}\resist(\edges(\Graph)\setminus\tree)}\\
  &= \frac{\sum_{\forest\in\forests{2}(\Graph)}\permute{i}{j}{k}{\ell}{\forest}\resist(\edges(\Graph)\setminus\forest)}{\sum_{\tree\in \trees(\Graph)}\resist(\edges(\Graph)\setminus\tree)},
  \end{align*}
 since $\permute{i}{j}{k}{\ell}{\forest}$ vanishes unless $F\in\mathcal{F}_{k\ell}$, in which case $\permute{i}{j}{k}{\ell}{\forest}$ equals $\indicator{i}(\forest)-\indicator{j}(\forest)$.
 Then all that remains is to show that $v$ induces the current assignment $\basis_k-\basis_\ell$, since then $v$ differs from $\laplace^+(\basis_k-\basis_\ell)$ by an overall constant, and hence
 \[\green(\Graph, \resist; \basis_i-\basis_j, \basis_k-\basis_\ell) = \laplace^+(\basis_k-\basis_\ell)_i-\laplace^+(\basis_k-\basis_\ell)_j = v_i-v_j\]
 as desired.

 We now prove that $v$ induces the correct currents.  
 First, let $e:i\to j$ be an edge. Then the current along this edge is 
 \[\current{e}{i}{j} = (v_i-v_j)/\resist(e) = \frac{\sum_{\forest\in\mathcal{F}_{k\ell}}\permute{i}{j}{k}{\ell}{\forest}\resist(\edges(\Graph)\setminus\forest)/\resist(e)}{\sum_{\tree\in \trees(\Graph)}\resist(\edges(\Graph)\setminus\tree)}.\]
 Now for each 2-forest $\forest$ with $\permute{i}{j}{k}{\ell}{\forest}$ nonzero, we find that $\forest\cup\{e\}$ is a spanning tree since $i$ and $j$ belong to different components of $\forest$. 
 A given spanning tree $\tree$ will arise in this way if and only if the path in $\tree$ from $k$ to $\ell$ contains the edge $e$; the coefficient $\permute{i}{j}{k}{\ell}{\forest}$ is then $+1$ if that path traverses $e$ in the direction from $i$ to $j$, or $-1$ if the direction is from $j$ to $i$. Below we will call this quantity $I_{\Graph|_\tree}(e:i\to j)$, since it is the current that would flow along edge $e$ if the current distribution $\basis_k-\basis_\ell$ were imposed on $\Graph|_\tree$ instead of $\Graph$. Moreover, if a spanning tree $T$ such that the path in $T$ from $k$ to $\ell$ contains the edge $e:i \to j$ corresponds in the above manner to a $2$-forest $\forest$ with $\permute{i}{j}{k}{\ell}{\forest}$ nonzero, then the identity
\[ \resist(\edges(\Graph) \setminus \tree) = \resist(\edges(\Graph)\setminus \forest)/\mu(e) \]
holds. We find
 \begin{equation}\label{current-average}
  \current{e}{i}{j} = \frac{\sum_{\tree\in\trees(\Graph)}\resist(\edges(\Graph)\setminus\tree)I_{\Graph|_\tree}(e:i\to j)}{\sum_{\tree\in \trees(\Graph)}\resist(\edges(\Graph)\setminus\tree)} \, . 
 \end{equation}
 In particular, the total current out of a given vertex $i$ is
 \[\sum_j \sum_{e:i\to j} \current{e}{i}{j} = \frac{\sum_{\tree\in\trees(\Graph)}\resist(\edges(\Graph)\setminus\tree)\left(\sum_j \sum_{e:i\to j}I_{\Graph|_\tree}(e:i\to j)\right)}{\sum_{\tree\in \trees(\Graph)}\resist(\edges(\Graph)\setminus\tree)}.\]
Now $\sum_j \sum_{e:i\to j}I_{\Graph|_\tree}(e:i\to j) = (\basis_k-\basis_\ell)_i$, and so is independent of $\tree$.
Therefore we have 
\[\sum_j \sum_{e:i\to j}\current{e}{i}{j} = (\basis_k-\basis_\ell)_i\]
as well, so the voltage distribution $v$ induces the current distribution $(\basis_k-\basis_\ell)$ as desired.
\end{proof}

As a consequence, we have the following homogeneity property of the Green's function:

\begin{prop} \label{proper-homogeneous} Let $\Graph$ be a connected graph, and let $\divD$ and $\divE$ be zero-sum elements of $\rr^{\vertices(\Graph)}$. The Green's function $\green(\Graph, \cdot\,;\divD,\divE)$ is homogeneous of weight one; that is, the equality 
\begin{equation}\label{homog-eq}
\green(\Graph,a\, \mu;\divD,\divE)=a \, \green(\Graph,\mu;\divD,\divE)
\end{equation}
holds for all $a \in \rr_{>0}$ and for all $\mu \in \rr_{>0}^{\mathrm{Ed}(\Gamma)}$. 
\end{prop} 
\begin{proof}
 Since $\divD$ and $\divE$ are each zero-sum vectors, they are linear combinations of vectors of the form $\basis_i-\basis_j$, so it suffices to show the homogeneity property in the special case of $\green(\Graph, \cdot\,; \basis_i-\basis_j, \basis_k-\basis_\ell)$.
 But this follows from Proposition~\ref{prop:tree-formula}, since spanning trees and 2-forests have fixed cardinalities, and these cardinalities differ by one.
 More precisely, every spanning tree for $\Graph$ contains $\#\vertices(\Graph)-1$ edges, and every $2$-forest contains $\#\vertices(\Graph)-2$ edges.
 Then the numerator of the formula in Equation \eqref{tree-formula} is a homogeneous polynomial of degree $\#\edges(\Graph)-\#\vertices(\Graph)+2$ in the resistances of the edges, and the denominator is homogeneous of degree $\#\edges(\Graph)-\#\vertices(\Graph)+1$.
 The ratio, therefore, is a homogeneous degree-$1$ rational function in the components of $\resist$.
\end{proof}

We will also show that in case $\divD=\divE$, the Green's function becomes \emph{concave}:
\begin{prop} \label{proper-concave}
Let $\Graph$ be a connected graph, and let $\divD$ be a zero-sum vector in $\rr^{\vertices(\Graph)}$. 
Then for all $\mu_1,\ldots,\mu_n \in \rr_{> 0}^{\edges(\Graph)}$ we have the inequality
\begin{equation}\label{concave-eq}
 \green\left(\Graph,\sum_{i=1}^n \mu_i;\divD,\divD\right) \geq \sum_{i=1}^n \green(\Graph,\mu_i;\divD,\divD).
\end{equation}
\end{prop}

In case $\divD=\basis_k-\basis_\ell$, Proposition~\ref{proper-concave} expresses the concavity of the effective resistance $r_\mathrm{eff}(k,\ell)$ as a function of the edge resistivities.  We will follow and generalize the proofs in \cite{me} and \cite{sh} of this special case.  The proofs use Jeans' Least Power Theorem, restated below in our terminology; see Theorem 357 on page 322 of \cite{je} for Jeans' original statement and proof.  The idea is that among all current patterns that have the correct totals at each vertex, the pattern that corresponds to some valid voltage assignment can be detected by virtue of its minimizing a quantity called the ``power dissipated''.

\begin{thm}[Jeans' Least Power Theorem]\label{jeans}  Let $(\Graph, \resist)$ be a proper resistive network, and let $\divD\in \rr^{\vertices(\Graph)}$. 
Let $I$ be a real-valued function on the set of oriented edges of $\Graph$ such that:
 \begin{enumerate}
  \item For all edges $e$ with endpoints $i$ and $j$, we have $I(e:i\to j) = -I(e:j\to i)$.
  \item For all vertices $i$, we have $\sum_{j}\sum_{e:i\to j} I(e:i\to j) = \divD_i$.
 \end{enumerate}
 Then the following are equivalent:
 \begin{enumerate}[label=\roman*.]
  \item There is a voltage assignment $v\in\rr^{\vertices(\Graph)}$ such that \[I(e:i\to j)=(v_i-v_j)/\resist(e),\] the current along edge $e$ from $i$ to $j$ induced by $v$.
  \item The function $I$ minimizes the \emph{power dissipated}
  \[\sum_{e\in\edges(\Graph)}\resist(e)I(e)^2,\]
  among all $I$ satisfying (1) and (2), where for an edge $e$ with endpoints $i$ and $j$, we write $I(e)^2$ for the quantity $I(e:i\to j)^2=I(e:j\to i)^2$.
 \end{enumerate}
\end{thm}
This theorem is relevant to us, because the minimal power dissipated $\sum_e\resist(e) I(e)^2$ consistent with a vertex current assignment $\divD$ is precisely the Green's function $\green(\Graph,\resist; \divD, \divD)$. 
Indeed, if $v=\laplace^+\divD$ and $I(e:i\to j)=(v_i-v_j)/\resist(e)$, then we have
\begin{align*}
 \green(\Graph,\resist;\divD,\divD) &= \transpose{\divD}\laplace^+\divD= \transpose{\divD}\,v\\
 &= \sum_{i\in\vertices(\Graph)}\divD_iv_i\\
 &= \sum_{i,j\in\vertices(\Graph)}\sum_{e:i\to j}I(e:i\to j) v_i,
\end{align*}
while we may write the power dissipated similarly:
\begin{align*}
 \sum_{e\in\edges(\Graph)} \resist(e)I(e)^2 &= \frac12\sum_{i,j\in\vertices(\Graph)}\sum_{e:i\to j}\resist(e)I(e:i\to j)^2\\
 &= \frac12\sum_{i,j\in\vertices(\Graph)}\sum_{e:i\to j}I(e:i\to j)(v_i-v_j)\\
 &= \frac12\sum_{i,j\in\vertices(\Graph)}\sum_{e:i\to j}I(e:i\to j) v_i - \frac12\sum_{i,j\in\vertices(\Graph)}\sum_{e:i\to j}I(e:i\to j) v_j.
\end{align*}
But by interchanging the dummy variables $i$ and $j$ in the second sum, we find that the second sum is the negative of the first.
Thus we obtain
\begin{equation}\label{green-power}
 \green(\Graph,\resist;\divD,\divD)=\sum_{e\in\edges(\Graph)}\resist(e)I(e)^2
\end{equation}
as desired.
We can now prove the concavity property of $\green(\Graph, \cdot\,; \divD, \divD)$.

\begin{proof}[Proof of Proposition \ref{proper-concave}]
 Given a resistance function $\resist$, denote by $I_\resist$ the edge current assignment consistent with vertex totals $\divD$ that minimizes the power dissipated $\sum_e \resist(e)I_\resist(e)^2$.  Then from the resistance functions $\resist_1,\dots,\resist_n$, we obtain $n$ current assignments $I_{\resist_1},\dots,I_{\resist_n}$, as well as the current assignment $I_\resist$ corresponding to the total $\resist:= \sum_{i=1}^n\resist_i$.  By Jeans' Least Power Theorem, we obtain the inequalities
 \[\sum_{e\in\edges(\Graph)}\resist_i(e) I_\resist(e)^2 \geq \sum_{e\in\edges(\Graph)}\resist_i(e) I_{\resist_i}(e)^2 = \green(\Graph,\resist_i; \divD,\divD)\]
 for each $i\in\{1,\dots,n\}$.  Summing over $i$, we obtain the inequality
 \[\green(\Graph,\resist;\divD,\divD) = \sum_{e\in\edges(\Graph)}\resist(e) I_\resist(e)^2 =\sum_{i=1}^n\sum_{e\in\edges(\Graph)}\resist_i(e) I_\resist(e)^2\geq \sum_{i=1}^n\green(\Graph,\resist_i;\divD,\divD),\]
 as desired.
\end{proof}

Jeans' Least Power Theorem also allows us to prove monotonicity, in the following sense:

\begin{prop}\label{proper-monotonic}
 Let $\Graph$ be a connected graph, and let $\divD$ be a zero-sum vector in $\rr^{\vertices(\Graph)}$. Let $\resist,\resist'\in\rr_{>0}^{\edges(\Graph)}$ be two resistance functions with $\resist\leq\resist'$. Then 
 \begin{equation}\label{monotonic-eq}
 \green(\Graph, \resist; \divD,\divD)\leq \green(\Graph, \resist'; \divD,\divD),
 \end{equation}
 with equality if and only if for each edge $e$, either $\resist(e)=\resist'(e)$ or no current flows along $e$ (in both $(\Graph,\resist)$ and $(\Graph,\resist')$). 
\end{prop}
\begin{proof}
 Let $I_\resist(e)^2$ and $I_{\resist'}(e)^2$ be the squared currents flowing along $e$ under resistances $\resist$ and $\resist'$, respectively.  Then Jeans' Least Power Theorem tells us that
 \[\sum_{e\in\edges(\Graph)}\resist(e) I_{\resist}(e)^2\leq\sum_{e\in\edges(\Graph)}\resist(e) I_{\resist'}(e)^2\leq\sum_{e\in\edges(\Graph)}\resist'(e) I_{\resist'}(e)^2,\]
 so $\green(\Graph, \resist; \divD,\divD)\leq \green(\Graph, \resist'; \divD,\divD)$ as desired.  
 
 If for each edge we have either $\resist(e)=\resist'(e)$ or $I_{\resist}(e)^2=0$, then we also have
 \begin{align*}
  \sum_{e\in\edges(\Graph)}\resist(e)I_\resist(e)^2 &=\sum_{e\in\edges(\Graph)}\resist'(e)I_\resist(e)^2\\
  &\geq \sum_{e\in\edges(\Graph)}\resist'(e)I_{\resist'}(e)^2
 \end{align*}
 by Jeans' Least Power Theorem, establishing the reverse inequality.
 
 Finally, suppose that $\sum_e\resist(e)I_\resist(e)^2=\sum_e\resist'(e)I_{\resist'}(e)^2$.  Then we must also have 
 \[\sum_{e\in\edges(\Graph)}\resist(e) I_{\resist}(e)^2=\sum_{e\in\edges(\Graph)}\resist(e) I_{\resist'}(e)^2,\]
 since the latter is sandwiched between two equal quantities.  Then $I_{\resist'}$ is an edge current assignment inducing $\divD$ and minimizing the power dissipated, so we must have $I_\resist=I_{\resist'}$.  We thus deduce
 \[\sum_{e\in\edges(\Graph)}\resist(e) I_{\resist}(e)^2=\sum_{e\in\edges(\Graph)}\resist'(e) I_{\resist}(e)^2,\]
 so $\sum_e(\resist'(e)-\resist(e))I_\resist(e)^2$=0.  Since each term is nonnegative, they must all be zero, so each edge $e$ either satisfies $\resist(e)=\resist'(e)$ or $I_\resist(e)=I_{\resist'}(e)=0$.
\end{proof}

Our next task is to show that the Green's function extends to the case of improper resistive networks, in the manner of Proposition \ref{continuous} and the following description.

\begin{proof}[Proof of Proposition \ref{continuous} and Equation (\ref{limiting-value})]
 Once again, we reduce to the case that $\divD=\basis_i-\basis_j$ and $\divE=\basis_k-\basis_\ell$ in order to use Equation~\eqref{tree-formula}.  
 We will divide the numerator and denominator of that formula by the same quantity, and then show that as $\resist\to\resist_0$ the new numerator and denominator converge to the correct limits.
 
 Recall that while spanning trees only exist for connected graphs, we may always consider \emph{maximal forests}, which consist of a spanning tree for each connected component of the graph.  
 Maximal forests of $\Graph|_S$ connect constructions on $\Graph/S$ to constructions on $\Graph$, as shown in the following proposition:
 
 \begin{prop}\label{max-forest}
  Let $\Graph$ be a graph. Let $S\subset\edges(\Graph)$ a collection of edges, and form the contracted graph $\Graph/ S$.  
  Let $F\subset\edges(\Graph)$ be another collection of edges, and suppose that $F\cap S$ is a maximal forest of $\Graph|_S$.  
  Let $i$ and $j$ be two vertices of $\Graph$.
  \begin{enumerate}
   \item Each path from $[i]$ to $[j]$ through $F\setminus S$ in $\Graph/S$ arises uniquely as the contraction of a path from $i$ to $j$ through $F$ in $\Graph$.
   \item The vertices $i$ and $j$ are in the same connected component of $\Graph|_F$ if and only if $[i]$ and $[j]$ are in the same connected component of $(\Graph/S)|_{F\setminus S}$.
   \item $F$ is cycle-free in $\Graph$ if and only if $F\setminus S$ is cycle-free in $\Graph/S$.
 \end{enumerate}
  \end{prop}

\begin{proof}
 We prove (1), from which we deduce (2) and (3).  
 \begin{enumerate}
  \item Let $[i]\xrightarrow{e_0} [i_1]\xrightarrow{e_1} \dots \xrightarrow{e_n} [j]$ be a path from $[i]$ to $[j]$ through $F\setminus S$ in $\Graph/S$.  
  In $\Graph$, the target endpoint $i_k'$ of $e_k$ may not agree with the source endpoint $i_k''$ of $e_{k+1}$.  
  However, these two vertices belong to the same connected component of $\Graph|_S$, so there is a unique path through $F\cap S$ from $i_k''$ to $i_k'$.  
  In this way we obtain the desired unique path $i\to j$ using only the $e_k$ and edges from $F\cap S$.
  
  \item Suppose that there is a path $P:i\to j$ through $F$ in $\Graph$.  
  After contracting the edges in $S$, we are left with a walk $P\setminus S: [i]\to [j]$ through $F\setminus S$ in $\Graph/S$.  
  This walk may not itself be a path, as some of the intermediate vertices may have been identified in the contracting process, but $P\setminus S$ will nevertheless contain a path $P':[i]\to [j]$.
  
  Conversely, any path from $[i]$ to $[j]$ through $F\setminus S$ in $\Graph/S$ extends to a path from $i$ to $j$ through $F$ in $\Graph$, by (1).
  
  \item Suppose that $F$ is not cycle-free, so that $F$ contains a nontrivial path, say $C:i\to i$.  
  Since $F\cap S$ is cycle-free, one of the edges of $C$ must belong to $F\setminus S$.  
  Then after contracting the edges of $S$, we obtain a nontrivial walk $C\setminus S: [i]\to [i]$, which must contain a nontrivial path $C':[i]\to [i]$. 
  Therefore $F\setminus S$ is not cycle-free in $\Graph/S$.
  
  Conversely, suppose that $F\setminus S$ contains a nontrivial cycle $[i]\to[i]$.  
  Then by (1), we obtain a nontrivial path $i\to i$ through $F$ in $\Graph$, so $F$ is also not cycle-free. \qedhere
 \end{enumerate}
\end{proof}

Denote the collection of edge sets forming maximal forests for $S$ by $\maxforests(S)$.  
The quantity $\Sigma$ by which we will divide the numerator and denominator of \eqref{tree-formula} is 
\[\Sigma := \sum_{\hphantom{\qquad}\mathclap{M\in \maxforests(S)}\hphantom{\qquad}}\resist(S\setminus M).\]

Consider first the denominator $\sum_{T\in\trees(\Graph)}\resist(\edges(\Graph)\setminus T)$.  
Each spanning tree $T$ intersects $S$ in a forest: this forest is either a maximal forest of $\Graph|_S$ or is strictly contained in one.  
If the former, then $F:= T\cap S$ is a maximal forest of $\Graph|_S$ and $T': = T\setminus S$ is a spanning tree for $\Graph/ S$. 
Conversely, if $F$ is any maximal forest of $\Graph|_S$ and $T'$ is any spanning tree for $\Graph/S$, then $T:=T'\cup F$ is a spanning tree for $\Graph$ by Proposition~\ref{max-forest}.  
Therefore we can write the denominator as
\[\begin{gathered} 
\sum_{T\in\trees(\Graph)}\resist(\edges(\Graph)\setminus T) = \sum_{\substack{T\in\trees(\Graph)\\\mathclap{T\cap S\in\maxforests(S)}}}\resist(\edges(\Graph)\setminus T)
+\sum_{\substack{T\in\trees(\Graph)\\\mathclap{T\cap S\notin\maxforests(S)}}}\resist(\edges(\Graph)\setminus T)\\
= \left(\sum_{T'\in\trees(\Graph/S)}\resist(\edges(\Graph/S)\setminus T')\right)\cdot\Sigma\quad  + \sum_{\substack{T\in\trees(\Graph)\\\mathclap{T\cap S\notin\maxforests(S)}}}\resist(\edges(\Graph)\setminus T),\text{ so}\\
\frac1\Sigma \sum_{T\in\trees(\Graph)}\resist(\edges(\Graph)\setminus T) 
= \sum_{T'\in\trees(\Graph/S)}\resist(\edges(\Graph/S)\setminus T') + \sum_{\substack{T\in\trees(\Graph)\\\mathclap{T\cap S\notin\maxforests(S)}}}\frac{\resist(\edges(\Graph)\setminus T)}\Sigma.
\end{gathered}
\]
The limit of the sum over $\tree'$, as $\resist\to\resist_0$, can be evaluated just by replacing $\resist$ with $\resist_0$; 
the result is nonzero since $\Graph/S$ has at least one spanning tree and $\resist_0$ takes only positive values on the edges of $\Graph/S$.  
The limit of the second term vanishes: if $T\cap S$ is a non-maximal forest of $\Graph|_S$, then it is strictly contained in some maximal forest $M$.  
Then we may bound $\resist(\edges(\Graph)\setminus T)/\Sigma$ by
\begin{align*}
 \frac{\resist(\edges(\Graph)\setminus T)}{\Sigma} &\leq \frac{\resist(\edges(\Graph)\setminus T)}{\resist(S\setminus M)} = \resist(\edges(\Graph)\setminus (S\cup T))\,\resist (M\setminus (T\cap S)).
\end{align*}
Now by assumption, $M\setminus (T\cap S)$ is nonempty and consists only of edges in $S$.  
Therefore, as $\resist\to\resist_0$, we have $\resist(M\setminus(T\cap S))\to 0$.  As a result, we have 
\[\lim_{\resist\to\resist_0} \left(\frac{1}\Sigma \sum_{\tree\in\trees(\Graph)}\resist(\edges(\Graph)\setminus\tree)\right) = \sum_{\tree'\in\trees(\Graph/S)}\resist_0(\edges(\Graph/S)\setminus \tree').
\]

The argument for the numerator is similar: 
again, we find that if a 2-forest $F$ intersects $S$ in a non-maximal forest of $\Graph|_S$, then $\resist(\edges(\Graph)\setminus F)/\Sigma$ tends to $0$ as $\resist$ approachs $\resist_0$.  
The $2$-forests $F$ which do intersect $S$ in a maximal forest have the property that $F\setminus S$ is a $2$-forest of $\Graph/S$, and conversely, every choice of $2$-forest for $\Graph/S$ and maximal forest for $\Graph|_S$ combine to give a $2$-forest for $\Graph$.  
So we find that 
\[\lim_{\resist\to\resist_0} \left(\frac{1}\Sigma \sum_{\forest\in\forests{2}(\Graph)}\permute{i}{j}{k}{\ell}{\forest}\resist(\edges(\Graph)\setminus\tree)\right) = \sum_{\forest'\in\forests{2}(\Graph/S)}\permute{[i]}{[j]}{[k]}{[\ell]}{\forest'}\resist_0(\edges(\Graph/S)\setminus \forest'),
\]
where we used Proposition~\ref{max-forest} again to obtain the identity $\permute{i}{j}{k}{\ell}{\forest} = \permute{[i]}{[j]}{[k]}{[\ell]}{\forest\setminus S}$.

All together, we have
\begin{align*}
 \green(\Graph,\resist; \basis_i-\basis_j, \basis_k-\basis_\ell) &= \frac{\sum_{\forest\in\forests{2}(\Graph)}\permute{i}{j}{k}{\ell}{\forest}\resist(\edges(\Graph)\setminus\forest)}{\sum_{\tree\in \trees(\Graph)}\resist(\edges(\Graph)\setminus\tree)} \\
 &= \frac{\sum_{\forest\in\forests{2}(\Graph)}\permute{i}{j}{k}{\ell}{\forest}\resist(\edges(\Graph)\setminus\forest)/\Sigma}{\sum_{\tree\in \trees(\Graph)}\resist(\edges(\Graph)\setminus\tree)/\Sigma}\\
 &\to \frac{\sum_{\forest'\in\forests{2}(\Graph/S)}\permute{[i]}{[j]}{[k]}{[\ell]}{\forest'}\resist_0(\edges(\Graph/S)\setminus\forest')}{\sum_{\tree'\in \trees(\Graph/S)}\resist_0(\edges(\Graph/S)\setminus\tree')}\\
 &= \green(\Graph/S,\resist_0|_{\edges(\Graph/S)}; \basis_{[i]}-\basis_{[j]}, \basis_{[k]}-\basis_{[\ell]})\\
 &= \green(\Graph/S,\resist_0|_{\edges(\Graph/S)}; [\basis_{i}-\basis_{j}], [\basis_{k}-\basis_{\ell}]).
\end{align*}
By linearity, we obtain that Equation~\eqref{limiting-value} holds for arbitrary zero-sum vectors $\divD,\divE\in\rr^{\vertices(\Graph)}$.  
And since a continuous function on a dense subset of a metric space extends to the whole space if and only if its limiting value is well-defined at each point, we have proven Proposition~\ref{continuous} as well.
\end{proof} 

We now deduce the full extent of Proposition~\ref{basicgreenfacts} as the following corollary:

\begin{cor} \label{homogeneousimproper}
 Let $\Graph$ be a connected graph, with $\divD$ and $\divE$ two zero-sum elements of $\rr^{\vertices(\Graph)}$.  Then the homogeneity equation \eqref{homog-eq} holds for all $a\in\rr_{\geq 0}$ and all $\resist\in \rr_{\geq 0}^{\edges(\Graph)}$, the concavity inequality \eqref{concave-eq} holds for all $\resist_1,\dots,\resist_n\in\rr_{\geq 0}^{\edges(\Graph)}$, and the monotonicity inequality \eqref{monotonic-eq} holds for all $\resist,\resist'\in\rr_{\geq 0}^{\edges(\Graph)}$.  
 
 If equality holds in the monotonicity inequality and $(\Graph, \resist')$ is proper, then for each edge $e$ either $\resist(e)=\resist'(e)$ or no current flows along $e$ in $(\Graph, \resist')$.
\end{cor}
\begin{proof}
 For homogeneity, the difference
 \[\green(\Graph,a\, \mu;\divD,\divE)-a \, \green(\Graph,\mu;\divD,\divE)\]
 vanishes for all $(a,\mu)$ in $\rr_{>0}\times\rr_{>0}^{\edges(\Graph)}$.  By continuity, then, it vanishes on the closure $\rr_{\geq 0}\times\rr_{\geq 0}^{\edges(\Graph)}$.  Similar arguments establish the concavity and monotonicity inequalities.
 
 Suppose that $(\Graph,\resist')$ is proper and $\resist\leq \resist'$.  Define a new resistance $\rho$ to be $(\resist+\resist')/2$; then $(\Graph,\rho)$ is also proper and $\resist\leq \rho\leq \resist'$.  By monotonicity, we have
 \[\green(\Graph,\resist;\divD,\divD)\leq \green(\Graph,\rho;\divD,\divD)\leq \green(\Graph,\resist';\divD,\divD),\]
 so if $\green(\Graph,\resist;\divD,\divD)=\green(\Graph,\resist';\divD,\divD)$ we must have $\green(\Graph,\rho;\divD,\divD)=\green(\Graph,\resist';\divD,\divD)$.  Then by Proposition \ref{proper-monotonic}, for each edge $e$ either $\rho(e)=\resist'(e)$, in which case $\resist(e)=\resist'(e)$, or no current flows through edge $e$ in $(\Graph,\rho)$ or $(\Graph,\resist')$.
\end{proof}

Our final result in this section is to prove the bound on the nonlinear part of $\green(\Graph, \cdot\,;\divD,\divE)$ asserted in Proposition~\ref{bound}:

\[ \left|\green\left(\Graph,\sum_{i=1}^n \resist_i;\divD,\divE\right)-\sum_{i=1}^n \green(\Graph,\resist_i;\divD,\divE)\right| \leq \|\divD\|\|\divE\| \min_{i\in\{1,\ldots,n\}} \sum_{j \neq i} |\mu_j|_1. \]

\begin{proof}[Proof of Proposition~\ref{bound}]
 We first prove the inequality in the case that $\divD=\divE$ and each $\resist_i$ takes only positive values.
 
 First consider the current assignment $\basis_k-\basis_\ell$.  
 In this case, no matter what the resistance function, the current flowing along any edge is between $-1$ and $1$ by Equation~\eqref{current-average}, since the latter expresses it as the weighted average of values in $\{-1,0,1\}$.  
 Therefore we have $I_\resist(e)^2\leq 1$ for every edge $e$ and every choice of resistance $\resist$.  
 A general zero-sum vector $\divD$ can be expressed as a nonnegative linear combination of vectors of the form $\basis_k-\basis_\ell$, with the total of all the coefficients equal to $\|\divD\|$.  
 Since the current along a given edge depends linearly on the vertex current assignment, we find that the currents are all bounded between $-\|\divD\|$ and $+\|\divD\|$, regardless of $\resist$.  
 In particular, we find that for all edges $e$ and resistances $\resist$, we have
 \[I_\resist(e)^2\leq \|\divD\|^2.\]
 Now set $\resist=\sum_{j=1}^n\resist_j$, and choose an arbitrary $i\in\{1,\dots,n\}$.  
 Then we have
 \begin{align*}
  \green(\Graph,\resist;\divD,\divD) &= \sum_{e\in\edges(\Graph)}\resist(e)I_{\resist}(e)^2\mathrlap{\text{ by Equation~\eqref{green-power}}}\\
  &\leq \sum_{e\in\edges(\Graph)}\resist(e)I_{\resist_i}(e)^2\text{ by Jeans' Least Power Theorem}\\
  &= \sum_{j=1}^n\sum_{e\in\edges(\Graph)}\resist_j(e) I_{\resist_i}(e)^2\\
  &= \sum_{e\in\edges(\Graph)}\resist_i(e) I_{\resist_i}(e)^2 + \sum_{j\neq i}\sum_{e\in\edges(\Graph)}\resist_j(e) I_{\resist_i}(e)^2\\
  &\leq \sum_{e\in\edges(\Graph)}\resist_i(e) I_{\resist_i}(e)^2 + \sum_{j\neq i}\sum_{e\in\edges(\Graph)}\resist_j(e)\|\divD\|^2\\
  &= \green(\Graph,\resist_i; \divD,\divD) + \|\divD\|^2\sum_{j\neq i}|\resist_j|_1.
 \end{align*}
Therefore we have 
\[\green(\Graph,\resist;\divD,\divD) - \sum_{j=1}^n \green(\Graph,\resist_j;\divD,\divD) \leq \green(\Graph,\resist;\divD,\divD) -  \green(\Graph,\resist_i;\divD,\divD)
 \leq\|\divD\|^2 \sum_{j\neq i}|\resist_j|_1,\]
and since this holds for each $i\in\{1,\dots,n\}$, we obtain
\[\green(\Graph,\resist;\divD,\divD)-\sum_{j=1}^n\green(\Graph,\resist_j;\divD,\divD)\leq \|\divD\|^2\min_{i\in\{1,\dots,n\}} \sum_{j\neq i}|\resist_j|_1.\]
 Since this holds for all $\resist_1,\dots,\resist_n\in\rr_{>0}^{\edges(\Graph)}$, by continuity it holds for improper resistances as well.
 
 Finally, we deduce from this the general case of the proposition. Note that for fixed $\Graph$ and $\resist_1,\dots,\resist_n\in\rr_{\geq 0}^{\edges(\Graph)}$, the function
 \[B:(\divD,\divE)\mapsto \green(\Graph,\resist;\divD,\divE)-\sum_{j=1}^n\green(\Graph,\resist_j;\divD,\divE)\]
 is a symmetric bilinear form that is positive semidefinite by the concavity inequality \eqref{concave-eq}, and we have proven that
 \[B(\divD,\divD)\leq c\|\divD\|^2\]
 with $c=\min_i\sum_{j\neq i}|\resist_j|_1$.
 By the Cauchy-Schwarz inequality, then, we obtain
 \[|B(\divD,\divE)|\leq\sqrt{B(\divD,\divD)}\sqrt{B(\divE,\divE)}\leq \sqrt{c\|\divD\|^2}\sqrt{c\|\divE\|^2}= c\|\divD\|\|\divE\|\]
 as desired.
\end{proof}

\vspace{0.5cm}

\noindent Address of the authors:\\ \\
Mathematical Institute  \\
Leiden University  \\
PO Box 9512  \\
2300 RA Leiden  \\
The Netherlands  \\ \\
Email: \verb+{bieselod,holmesdst,rdejong}@math.leidenuniv.nl+

\end{document}